\documentclass{amsart}
\usepackage{amsmath,amsthm,amssymb}
\usepackage[backend=biber,style=alphabetic,isbn=false,doi=false,url=false,eprint=false,maxnames=6,minalphanames=4,maxalphanames=6]{biblatex} 
\renewbibmacro{in:}{} 
\usepackage{graphicx} 
\usepackage{tikz} 
\usetikzlibrary{arrows,decorations.markings,shapes}

\usepackage{hyperref} 
\hypersetup{colorlinks,citecolor=blue}
\usepackage[nameinlink]{cleveref} 
\usepackage[hyphenbreaks]{breakurl} 

\newcommand{\ZZ}{\mathbb Z}
\newcommand{\RR}{\mathbb R}
\newcommand{\eps}{\varepsilon}
\newcommand{\inv}{^{-1}}

\newcommand{\Out}{\mathrm{Out}}

\newtheorem{theorem}{Theorem}[section]
\newtheorem{lemma}[theorem]{Lemma}
\newtheorem{corollary}[theorem]{Corollary}
\newtheorem{proposition}[theorem]{Proposition}

\theoremstyle{definition}
\newtheorem{definition}[theorem]{Definition}

\theoremstyle{remark}

\newtheorem{example}[theorem]{Example}

\title{Axis Bundles in Free-by-Cyclic Groups}
\author{Maxwell Plummer}
\address{Department of Mathematics, Rice University, Houston, TX}
\email{msp15@rice.edu}

\begin{document}

\begin{abstract}
    Given a splitting of a free-by-cyclic group, the associated monodromy acts on outer space preserving Handel and Mosher's ``axis bundle." We show that the property of a monodromy having a ``lone axis'' is non-generic in the sense that the associated splittings are projectively discrete in first cohomology. Additionally, we show that this discreteness statement cannot be promoted to a finiteness statement.
\end{abstract}

\thanks{The author was partially supported by NSF DMS-1745670.}

\maketitle

\section{Introduction}
The action of $\Out(F_N)$ on Culler and Vogtmann's outer space $CV_N$ \cite{CV86} is often understood by analogy with the action of the mapping class group of a surface on its Teichmüller space. The dynamically interesting mapping classes are those that are pseudo-Anosov, and the analog in $\Out(F_N)$ are the fully irreducible automorphisms, as studied by Bestvina--Handel \cite{BH92}. Pseudo-Anosov homeomorphisms translate along unique axes in Teichmüller space, while fully irreducible automorphisms may have multiple axes. For a nongeometric fully irreducible automorphism $\varphi:F_N\to F_N$, Handel and Mosher defined the {\em axis bundle} $\mathcal{A}_\varphi \subseteq CV_N$ which contains potentially many $\varphi$-invariant geodesics \cite{hm11axis_bundle}. Mosher and Pfaff \cite{mp16lone-axes} characterized when a free group automorphism has a {\em lone axis}.

The mapping torus of a pseudo-Anosov homeomorphism is a fibered 3-manifold $M$, and the fiber is dual to a class that lies in the cone over a face of the (polyhedral) Thurston normal ball in $H^1(M;\RR)$, see \cite{thurston-norm,Fried-cross-section} and \cite[Exposé 14]{FLP}. The primitive integral classes in this cone are all dual to fibers whose monodromy homeomorphisms are dynamically related via the suspension flow on $M$.
Fried \cite{Fried-stretch-factor} proved that the logarithm of their stretch factors extends to a homogeneous, convex function on this cone; this function can be computed by McMullen’s Teichmüller polynomial \cite{McMullen2000}.

For a free-by-cyclic group $G = G_\varphi = F_N\rtimes_\varphi \ZZ$, the cone over a component of the BNS invariant \cite{BNS87} is the analog of the cone over a fibered face. If $\varphi$ is fully-irreducible, then so is the monodromy for any splitting defined by a point of the cone by work of Dowdall--Kapovich--Leininger \cite{DKL3,DKL2}, who further related various features of these monodromies, see also \cite{AKHR}. In this paper, we prove that having a lone axis is {\em not} a feature of the monodromies that is shared.

\begin{theorem} \label{thm:main-theorem-intro-version}
    Let $\varphi$ be a fully irreducible nongeometric automorphism of $F_N$ with a lone axis in outer space. Let $G = F_N\rtimes_\varphi \ZZ$ and $r^*\in H^1(G;\RR)$ the class whose associated monodromy is $\varphi$. There is an open cone $C\subseteq H^1(G;\RR)$ containing $r^*$ such that the only lone axis monodromy associated to a primitive integral point of $C$ is $\varphi$. Alternatively, classes with lone axis monodromies are discrete in the projectivization $H^1(G;\RR)\setminus \{0\}/\RR_+$.
\end{theorem}

In fact, our methods lend themselves to the proof of a stronger theorem. In \Cref{section:proof-main-theorem} we show that splittings for which the monodromy has axis bundle of dimension at most any fixed number $k$ are also discrete in $H^1(G;\RR)\setminus \{0\}/\RR_+$, see \Cref{thm:projectively-discrete}. Additionally, we show that in general it is impossible to improve the statement from a discreteness result to a finiteness result. Indeed, we explicitly construct a free-by-cyclic group $G$ that has infinitely many lone axis monodromies among splittings of $G$, see \Cref{thm:extended-example}.

The relationship between 3-manifolds and free-by-cyclic groups can be seen in Dowdall--Kapovich--Leininger's \textit{folded mapping torus} $X$ associated to $[\varphi]\in\Out(F_N)$ \cite{DKL1}. The space $X$ has a suspension semiflow, which dynamically relates splittings of $G$ in the same manner as for 3-manifolds. It was using this machinery that Dowdall--Kapovich--Leininger were able to prove invariance of full irreducibility across a component of $\RR_+\cdot \Sigma(G)$. This was extended to all splittings of $G$ by Mutanguha using other methods \cite{mut21}.

One interpretation of \Cref{thm:main-theorem-intro-version} is that lone axis automorphisms are non-generic among monodromies of free-by-cyclic groups $G$ with $\mathrm{dim} (H^1(G;\RR)) \geq 2$. This is in contrast to a result of Kapovich--Pfaff that there is a `train track directed' random-walk in $\Out(F_N)$ under which lone axis automorphisms are generic \cite{KP15-rand-walk}. The question of whether lone axis automorphisms are generic for other random walks appears to be open, though the authors of \cite{KMPT-random-walk} believe the answer should be negative.

The original motivation leading to this paper was an attempt to construct a canonical folding sequence associated to a component of the cone $\RR_+\cdot (\Sigma (G) \cap -\Sigma(G))$. This was intended to form some sort of analog of the {\em veering triangulations} constructed by Agol \cite{agol2011}; the construction is defined in terms of a canonical splitting sequence associated to a monodromy. These triangulations depend only on the fibered face of the Thurston norm ball. To obtain such a free-by-cyclic group analog, the hope was that the property of having a lone axis was invariant within this component, which turned out to be far from the case. For a given nongeometric fully irreducible $[\varphi]\in\Out(F_N)$, Pfaff and Tsang have very recently constructed a finite collection of canonical fold lines associated to $\varphi$ \cite{pfaff-tsang-cubist}. The author intends to investigate the possibility of a canonical folding sequence associated to a cone in future work.

In \Cref{section:background}, we give necessary background on $\Out(F_N)$, the axis bundle, and the folded mapping torus. We also make a slight modification to the cell structure on the folded mapping torus. We prove \Cref{thm:main-theorem-intro-version} and its generalization \Cref{thm:projectively-discrete} in \Cref{section:proof-main-theorem}. Finally, in \Cref{section:extended-example} we show that the projective discreteness of \Cref{thm:main-theorem-intro-version} cannot be promoted to finiteness by considering a particular example in some detail.

\textbf{Acknowledgments.} The author would like to thank their Ph.D. advisor Chris Leininger for suggesting the original version of this problem and for providing helpful feedback on the numerous versions of this paper. Thanks also to Robbie Lyman for a helpful comment on an earlier draft.

\section{Background} \label{section:background}

\subsection{Free Group Automorphisms}

Let $F_N$ be the free group of rank $N$. An automorphism $\varphi$ of $F_N$ is {\em fully irreducible} if no nontrivial conjugacy class of a proper free factor is periodic under $\varphi$. It is {\em atoroidal} if there is no $n\geq 1$ and $1\neq w\in F_N$ such that $\varphi^n(w)$ is conjugate to $w$. Some automorphisms are induced by homeomorphisms of punctured surfaces; these are the \textit{geometric} automorphisms. Automorphisms which are not geometric are called \textit{nongeometric}. It is a result of Bestvina--Handel \cite{BH92} that a fully irreducible automorphism is atoroidal if and only if it is not induced by a pseudo-Anosov homeomorphism on a surface with one boundary component.

Associated to a fully irreducible automorphism $\varphi$ is the attracting tree $T_+$ and the repelling tree $T_-$. A result proved independently by Handel--Mosher and Guirardel \cite{hm07parageom,guirardel-core} is that $\varphi$ is geometric if and only if both $T_+$ and $T_-$ are geometric in the sense that they are dual to a measured lamination on a 2-complex. The class of fully irreducible nongeometric automorphisms is divided into the subclasses {\em parageometric} and {\em ageometric} in terms of the atracting tree $T_+$, see \cite{hm07parageom}. The parageometric outer automorphisms are the nongeometric outer automorphisms that have geometric attracting trees. The ageometric automorphisms are the fully irreducible nongeometrics that are not parageometric. Note that by the result mentioned above, the inverse of a parageometric automorphism is ageometric.

\subsection{Graph Maps}

We follow conventions of \cite{BH92} for this section, with some modifications as in \cite{DKL1} for the inclusion of valence two vertices. We consider connected finite graphs $\Gamma$ without valence one vertices. Denote by $V(\Gamma), E(\Gamma)$ the vertex set and edge set of $\Gamma$, respectively. A {\em graph map} $f: \Gamma\to\Gamma$ sends vertices to vertices and maps edges across edge paths.

A {\em train track map} $f:\Gamma\to\Gamma$ is a graph map such that $f^n$ is locally injective on the interior of all edges $e\in E(\Gamma)$ for all $n\geq 1$. An alternative description of this condition is in terms of {\em illegal turns}, which we now describe. A {\em direction} at a vertex $v\in \Gamma$ is a germ of an oriented edge incident to $v$. The map $f$ induces a map on directions $Df$; we call a direction {\em periodic} if it is fixed by some power of $Df$. A {\em turn} at $v$ is an unordered pair of directions; a turn is {\em degenerate} if it consists of the same direction twice. We call a turn $\{e_i,e_j\}$ {\em illegal} if $\{Df^n (e_i), Df^n (e_j)\}$ is degenerate for some $n\geq 1$. A graph map $f$ is a train track map if it does not map the interior of any edge across an illegal turn. We call an edge path in $\Gamma$ {\em legal} if does not cross an illegal turn.

Identify $F_N$ with $\pi_1(R_N)$, where $R_N$ is the graph with one vertex and $N$ edges. A {\em marked graph} is a graph $\Gamma$ with a homotopy equivalence $\iota: R_N\to \Gamma$ which we call a {\em marking}. We say a train track map $f: \Gamma\to \Gamma$ represents an automorphism $\varphi: F_N\to F_N$ if $f$ induces $\varphi$ under the marking (this is only defined up to conjugacy in $F_N$, so we typically consider representatives of outer automorphisms).

If we enumerate the edges of $\Gamma$ by $e_1, \dots, e_k$, we can associate to a graph map $f$ a transition matrix $A$. The entry $a_{ij}$ is the number of times $e_i$ crosses over $e_j$ under $f$, not considering orientations. If for any pair $(i,j)$ there exists $n$ such that the $(i,j)$-entry of $A^n$ is positive, we say $A$ is irreducible. A train track map is irreducible if its transition matrix is irreducible. Geometrically, for each $i,j$, there is an $n$ such that $f^n(e_i)$ maps over $e_j$.

Suppose $f:\Gamma\to\Gamma$ is an irreducible train track map, the transition matrix $A$ has a unique positive (left) eigenvector $(x_i)$ up to scaling that corresponds to the largest eigenvalue $\lambda$ of $A$. We call $\lambda$ the {\em stretch factor} of $f$. If $\lambda > 1$, $f$ is called an {\em expanding irreducible} train track map. Any train track map representing a fully irreducible automorphism is expanding irreducible.

\begin{example}[Running Example] \label{ex:automorphism}
    We introduce an example which we will consider throughout this section and in more depth in \Cref{section:extended-example}. Let $\varphi: F_3\to F_3$ be the automorphism defined by $a\mapsto ca, b\mapsto ab, c\mapsto b\inv a b$. An expanding irreducible train track representative $f$ for $\varphi$ is depicted in \Cref{fig:train-track-representative}; the graph $\Gamma$ is marked by a map from $R_3$ given by $a\mapsto ea, b\mapsto bd\bar{e}, c\mapsto e\bar{d} c d \bar{e}$. The automorphism $\varphi$ is ageometric and fully irreducible, which was verified in SAGE using the package \cite{coulbois-train-track} via the algorithm for detecting full irreducibility of \cite{kap14iwip}.
\end{example}

\begin{figure}
    \centering
\begin{tikzpicture}[scale=0.8]

    \draw[decoration={markings, mark=at position 0.5 with {\arrow{<}}},
    postaction={decorate}
    ](-3,-1) arc[start angle=-90, end angle = 90, radius = 1cm];
    \node at (-2,0) [right]{e};

    \draw[decoration={markings, mark=at position 0.5 with {\arrow{<}}},
    postaction={decorate}
    ](-3,1) arc[start angle=90, end angle = 270, radius = 1cm];
        \node at (-4,0) [left]{a};

    \draw[decoration={markings, mark=at position 0.5 with {\arrow{>}}}, postaction={decorate}] (-3,1) .. controls (-3,2) and (0,1) .. (0,0);
        \node at (-1.5, 1.2) [above]{b};

    \draw[decoration={markings, mark=at position 0.5 with {\arrow{<}}},
    postaction={decorate}] (-3,-1) .. controls (-3, -2) and (-1, 0) .. (0,0);
        \node at (-1.5, -0.6) [below]{d};
    
    \draw[decoration={markings, mark=at position 0.0 with {\arrow{>}}},
        postaction={decorate}] (1,0) circle(1cm);
        \node at (2,0)[right]{c};

    \node at (-3,1)[circle,fill,red, minimum size = 0.2cm, inner sep = 0pt]{};
    \node at (-3, -1)[circle,fill, minimum size = 0.2cm, inner sep = 0pt]{};
    \node at (0,0)[circle,fill,blue, minimum size = 0.2cm, inner sep = 0pt]{};
\end{tikzpicture}
\qquad
\begin{tikzpicture}[scale=0.8]

    \node at (-2,0)[right]{d};
    \draw[decoration={markings, mark=at position 0.5 with {\arrow{>}}},
    postaction={decorate}
    ](-3,-1) arc[start angle=-90, end angle = 90, radius = 1cm];

        \draw[decoration={markings, mark=at position 0.7 with {\arrow{<}}},
        postaction={decorate}
        ](-3,1) arc[start angle=90, end angle = 135, radius = 1cm];
        
        \draw[decoration={markings, mark=at position 0.5 with {\arrow{<}}},
        postaction={decorate}
        ]({cos(135)-3},{sin(135)}) arc[start angle=135, end angle = 180, radius = 1cm];
        \draw[decoration={markings, mark=at position 0.5 with {\arrow{<}}},
        postaction={decorate}
        ](-4,0) arc[start angle=180, end angle = 225, radius = 1cm];
        \draw[decoration={markings, mark=at position 0.3 with {\arrow{<}}},
        postaction={decorate}
        ]({cos(225)-3},{sin(225)}) arc[start angle=225, end angle = 270, radius = 1cm];
        \node at (-3.89,0.81) [above]{e};
        \node at (-3.96,0.31) [left]{a};
        \node at (-3.96,-0.31) [left]{d};
        \node at (-3.89,-0.81) [below]{c};
    
    \draw[decoration={markings, mark=at position 0.5 with {\arrow{>}}}, postaction={decorate}] (-3,1) .. controls (-3,2) and (0,1) .. (0,0);
        \node at (-1.5, 1.2) [above]{a};
    \draw[decoration={markings, mark=at position 0.5 with {\arrow{<}}},
    postaction={decorate}] (-3,-1) .. controls (-3, -2) and (-1, 0) .. (0,0);
        \node at (-1.5, -0.6) [below]{b};

    \draw[decoration={markings, mark=at position 0.5 with {\arrow{>}}},
        postaction={decorate}] (0,0) arc[start angle = -180, end angle = 0, radius = 1cm];
    \draw[decoration={markings, mark=at position 0.5 with {\arrow{>}}},
        postaction={decorate}] (2,0) arc[start angle = 0, end angle = 180, radius = 1cm];
        \node at (1,1)[above]{a};
        \node at (1,-1)[below]{e};

    \node at (-3,1)[circle,fill,red, minimum size = 0.2cm, inner sep = 0pt]{};
    \node at (0,0)[circle,fill,blue, minimum size = 0.2cm, inner sep = 0pt]{};
    \node at (-3, -1)[circle,fill, minimum size = 0.2cm, inner sep = 0pt]{};
\end{tikzpicture}

\caption{Train track representative $f:\Gamma\to\Gamma$ of the automorphism $\varphi$ from \Cref{ex:automorphism}. The graph on the left is the graph $\Gamma$. On the right is the same graph, now labeled with the images of the edges under $f$. Note that this map has a unique illegal turn $\{\bar{b},\bar{c}\}$ at the blue vertex.}
    \label{fig:train-track-representative}
\end{figure}
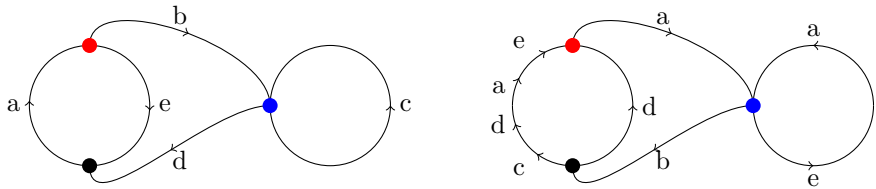

\subsection{Stallings Fold Decompositions} \label{subsec:stallings-fold}
Let $f: \Gamma\to\Gamma$ be a train track map. We denote by $\Gamma_0$ the subdivision of $\Gamma$ obtained by adding vertices to each point of $f\inv(V(\Gamma))$. We refer to $\Gamma_0$ as the {\em Stallings subdivision} of $\Gamma$. We label each oriented subdivided edge $e\in E(\Gamma_0)$ by the oriented image of $e$ under $f$.

Suppose two edges $e_1,e_2\in E(\Gamma_0)$ have a common initial vertex $v$ and share the oriented edge label $a$. That is, $f(e_1) = f(e_2) = a$ and in particular $\{e_1,e_2\}$ is an illegal turn. In this case we say $e_1$ and $e_2$ `carry the same edge label.' Then we can define a quotient map $q_1: \Gamma_0\to \Gamma_1$ which identifies $e_1$ and $e_2$ preserving the edge labels. We call such a map a {\em fold}. The {\em Stallings fold decomposition} of $f$ \cite{stallings-fold} is the factorization
\[
    \Gamma\xrightarrow{\pi}\Gamma_0 \xrightarrow{q_1} \Gamma_1 \xrightarrow{q_2} \dots \xrightarrow{q_k} \Gamma_k \xrightarrow{h} \Gamma
\]
where each $q_i$ is a single fold, $h: \Gamma_k\to \Gamma$ is a homeomorphism, and $\pi:\Gamma\to\Gamma_0$ is the `identity' map to the subdivision. Explicitly, $f = h q_k \cdots q_1 \pi$. Note that the sequence of folds $q_i$ is in general not unique.

\begin{example}
The train track map from \Cref{ex:automorphism} has a folding sequence consisting of four folds. Viewing the graph on the right of \Cref{fig:train-track-representative}, there is an illegal turn at the blue vertex, where there are two incoming $a$ edges, so these may be folded together. The remaining folds in the sequence identify $e$ edges, $a$ edges again, then $d$ edges. In this case, the folding sequence is unique since there is only one choice of fold at any stage.
\end{example}

\subsection{Outer Space}
A metric graph is a graph $\Gamma$ with a path metric whose edges are locally isometric images of intervals $I\subseteq \RR$; alternatively, it is an assignment of a positive length to each edge of $\Gamma$, encoded by a map $\ell: E(\Gamma) \to \RR_+$. The {\em volume} of $\Gamma$ is the sum of the edge lengths. Outer space, denoted $CV_N$, is the space of equivalence classes of marked metric graphs $\Gamma$ of volume one, we denote the marking $\iota:R_N\to \Gamma$ (recall $\iota$ is a homotopy equivalence). Alternatively, by taking universal covers, $CV_N$ is sometimes thought of as a space of actions of $F_N$ on metric simplicial trees. Outer space was introduced by Culler and Vogtmann in \cite{CV86}, see also the survey article by Vogtmann \cite{Vog15}. The space $CV_N$ has a nice simplicial structure (with missing faces). Each marked topological graph $\Gamma$ without valence two vertices defines a simplex of dimension $|E(\Gamma)| - 1$, where the correspondence is via the possible volume one length functions $\ell: E(\Gamma)\to \RR_+$. The faces of a simplex correspond to an assignment of length 0 to some edges of $\Gamma$ and the `missing faces' are those which would have assigned length 0 to a loop of $\Gamma$.

If $f:\Gamma\to\Gamma$ is a train track map representing a fully irreducible automorphism $\varphi$, then $f$ induces a natural metric structure on $\Gamma$. If $(x_i)$ is a positive eigenvector of the transition matrix $A$, then we assign the length $x_i$ to the edge $e_i$ and scale so that the volume of $\Gamma$ is one. We call this metric the \textit{eigenmetric} induced by $f$. On $\Gamma$ with this metric we can take $f$ to expand legal paths uniformly according to the stretch factor $\lambda$. When $\Gamma$ has been endowed with this metric and $f$ expands edges uniformly by $\lambda$ we call $f$ a \textit{metric train track map}.

A common metric studied on $CV_N$ is the {\em Lipschitz Metric} \cite{white-thesis,FM11lipschitz}, which is non-symmetric. If $(\Gamma,\iota), (\Gamma',\iota')\in CV_N$, then the Lipschitz distance between them is $d(\Gamma,\Gamma') = \inf_{f} \log(L(f))$, where $f: \Gamma\to \Gamma'$ satisfies $f\iota = \iota'$ and $L(f)$ is the Lipschitz constant of $f$. The outer space $CV_N$ with the Lipschitz metric is geodesic; one class of geodesics of $CV_N$ with this metric are the {\em folding paths} or {\em fold lines}. Locally, traveling along a fold line means identifying increasingly long segments of certain edges at the same vertex together; these arise naturally in the case that those edges form an illegal turn of a train track representative $f$ of an automorphism $\varphi: F_N\to F_N$. In this case, a fold line can be seen as a sort of continuous version of a Stallings fold decomposition.

The group $\Out(F_N)$ has a right action on $CV_N$ by change of markings. If $[\varphi]\in \Out(F_N)$, we can abuse notation by considering $\varphi$ to be a graph map $R_N\to R_N$. Then the action is defined by $(\Gamma, \iota)\cdot [\varphi] = (\Gamma, \iota\varphi)$. This action is an isometric action with respect to the Lipschitz metric, and in particular it preserves the fold lines. In the trees characterization of $CV_N$, the nonsimplicial trees $T_+$ and $T_-$ can be considered to live in $\partial CV_N$ \cite{CM87,Paulin}.

\subsection{Axis Bundles and Ideal Whitehead Graphs} \label{sec:axis-bundle}
For a nongeometric fully irreducible $\varphi\in\Out(F_N)$, let $TT(\varphi)\subseteq CV_N$ be the set of marked metric graphs on which there exists a metric train track representative of $\varphi$. The following is one of three equivalent definitions for the axis bundle introduced in \cite{hm11axis_bundle}.

\begin{definition}
    The {\em axis bundle} $\mathcal{A}_\varphi$ of a nongeometric automorphism $\varphi$ is
    the closure of $\bigcup_{n=1}^\infty TT(\varphi^n)$.
\end{definition}

Another characterization of $\mathcal{A}_\varphi$ given in \cite{hm11axis_bundle} is the union of all fold lines in $CV_N$ that limit to $T_+,T_-\in \partial CV_N$ in forwards and backwards time, respectively. Under that framework, one can see that $\mathcal{A}_\varphi$ is actually a collection of geodesic axes invariant under the action of $\varphi$.

An automorphism $\varphi$ is said to have a {\em lone axis} if $\mathcal{A}_\varphi$ is homeomorphic to $\mathbb{R}$. Handel and Mosher prove in \cite{hm11axis_bundle} that parageometric automorphisms do not have lone axes, so the study of lone axis automorphisms is reduced to the ageometric setting. 

In \cite{mp16lone-axes}, Mosher and Pfaff give necessary and sufficient conditions for an automorphism to have a lone axis in terms of the `ideal Whitehead graph.' We now describe this graph and these conditions and will use them to verify that the automorphism of \Cref{ex:automorphism} has a lone axis.

Let $\varphi: F_N\to F_N$ be a nongeometric fully irreducible automorphism with train track representative $f:\Gamma\to\Gamma$. A {\em Nielsen path} in $\Gamma$ is an edge path $\rho$ such that $f(\rho)$ is homotopic to $\rho$ rel endpoints; an {\em indivisible} Nielsen path is a Nielsen path that cannot be written as a nontrivial concatenation of Nielsen paths. Likewise, a {\em periodic Nielsen path} is an edge path that is a Nielsen path of $f^n$ for some $n$, as in $f^n(\rho)$ is homotopic rel endpoints to $\rho$, and a {\em periodic indivisible Nielsen path} is a periodic Nielsen path that cannot be written as a nontrivial concatenation of periodic Nielsen paths. There are only finitely many periodic indivisible Nielsen paths in $\Gamma$ \cite[Lemmas 4.2.5-6]{BFH00}, so after a possible subdivision of $\Gamma$, we assume that each such endpoint is in $V(\Gamma)$. After such a subdivision, we call a vertex $v\in V(\Gamma)$ {\em principal} if $v$ has at least three periodic directions or $v$ is the endpoint of a periodic Nielsen path.

We will now give a description of the ideal Whitehead graph of $\varphi$ which will serve as a definition for us in the special case that $f$ admits no periodic Nielsen paths. Let $v\in V(\Gamma)$ be a principal vertex. The {\em local Whitehead graph at $v$}, $\mathcal{LW}(\Gamma,v)$, has a vertex for each direction at $v$ and an edge connecting directions when that turn is taken, that is, an edge of $\Gamma$ maps over it under $f^n$ for some $n\geq 1$. The {\em local stable Whitehead graph at $v$,} $\mathcal{SW}(\Gamma,v)$, is the full subgraph of $\mathcal{LW}(\Gamma,v)$ with vertices the periodic directions at $v$. The {\em ideal Whitehead graph} of $\varphi$, $\mathcal{IW}(\varphi)$, is a union of $\mathcal{SW}(\Gamma,v)$ over all $v\in V(\Gamma)$; when $f$ has no periodic Nielsen paths, the union is a disjoint union.

The above description suffices for our purposes, but we briefly recall Handel and Mosher's definition \cite{hm11axis_bundle}, see also \cite[Section 2.9]{pfaff-thesis}. Starting by `realizing' the leaves of the {\em expanding lamination} \cite{BFH97} of $\varphi$ in the attracting tree $T_+$, the graph $\mathcal{IW}(\varphi)$ records how leaves of the lamination cross branch points of $T_+$. Let $p:\tilde \Gamma \to \Gamma$ be the universal cover, the train track map $f:\Gamma\to\Gamma$ determines a map $f_+: \tilde{\Gamma}\to T_+$ which may identify distinct points $\tilde{v},\tilde{w}\in \tilde{\Gamma}$. If this map also identifies directions at $\tilde{v}$ and $\tilde{w}$, then the corresponding vertices of $\mathcal{SW}(\Gamma,p(\tilde v))$ and $\mathcal{SW}(\Gamma,p(\tilde w))$ are identified inside of $\mathcal{IW}(\varphi)$. Lemma 3.1 of \cite{hm11axis_bundle} implies that this situation only occurs when there is an indivisible Nielsen path between $p(\tilde v)$ and $p(\tilde w)$.

Suppose $\mathcal{IW}(\varphi)$ has $k$ components, and the $i$th component has $m_i$ vertices. We define the {\em rotationless index} of $\varphi$ to be the sum
\[
    i(\varphi) = \sum_{i=1}^k 1 - \frac{m_i}{2}.
\]
We can now state the lone axis criteria of \cite{mp16lone-axes}.

\begin{theorem}[\cite{mp16lone-axes}, Theorem 4.7] \label{thm:mp16-lone-axis-criteria}
    Let $\varphi\in \Out(F_N)$ be ageometric and fully irreducible. The automorphism $\varphi$ has a lone axis in outer space if and only if the following conditions hold:
    \begin{enumerate}
        \item The rotationless index $i(\varphi) = 3/2 - N$,
        \item No component of the ideal Whitehead graph has a cut vertex.
    \end{enumerate}
\end{theorem}

\begin{example} \label{ex:lone-axis}
We show that the automorphism $\varphi$ of our running example has a lone axis. Using the package \cite{coulbois-train-track}, we verified that the train track representative $f:\Gamma\to \Gamma$ of \Cref{fig:train-track-representative} has no periodic Nielsen paths. We proceed by analyzing \Cref{fig:train-track-representative}. The turn $\{\bar e, a\}$ at the black vertex is crossed by the edge $c$ under $f$. We can see that the direction $\bar c$ is not in the image of $Df$, and the remaining nine directions are permuted transitively. Similarly, the turns consisting of periodic directions are permuted transitively. Thus, from the single taken turn $\{\bar e, a\}$, we have that every turn consisting of periodic directions is taken by a power of $f$. This implies that the local stable Whitehead graphs are triangles and the ideal Whitehead graph is the disjoint union of three triangles. Then
\[
    i(\varphi) = \sum_{j=1}^3 1- \frac{3}{2} = -\frac{3}{2} = \frac{3}{2} - 3
\]
and no component of $\mathcal{IW}(\Gamma)$ has a cut vertex, so $\varphi$ has a lone axis by \Cref{thm:mp16-lone-axis-criteria}.
\end{example}

\subsection{Folded Mapping Torus}

A {\em free-by-cyclic group} 
\[
    G = G_\varphi = F_N\rtimes_\varphi \ZZ = \langle w, r | r\inv w r = \varphi(w) \ \forall w\in F_N \rangle
\]
is a semidirect product of $F_N$ with $\ZZ$. We sometimes refer to such a group as the {\em mapping torus group} of $\varphi$. We now describe an associated $K(G,1)$ space which has proven fruitful in studying the splittings of free-by-cyclic groups.

\begin{figure}
    \centering
    \begin{tikzpicture}
        \draw[decoration={markings, mark=at position 0.5 with {\arrow{>}}},postaction={decorate}] (0,0) -- (4,0);
            \draw[decoration={markings, mark=at position 0.5 with {\arrow{>}}},postaction={decorate}] (0,4) -- (1,4);
            \draw[decoration={markings, mark=at position 0.5 with {\arrow{>}}},postaction={decorate}] (1,4) -- (2,4);
            \draw[decoration={markings, mark=at position 0.5 with {\arrow{>}}},postaction={decorate}] (2,4) -- (3,4);
            \draw[decoration={markings, mark=at position 0.5 with {\arrow{>}}},postaction={decorate}] (3,4) -- (4,4);
        \draw[red] (4,0) -- (4,4);
        \draw (0,0) -- (0,4);
        \draw (4,1) -- (1,4);
        \draw[blue] (3,2) -- (3,4);
        \draw[red] (2,3) -- (2,4);
        \node at (2,0)[below]{a};
        \node at (0.5,4)[above]{c};
        \node at (1.5,4)[above]{d};
        \node at (2.5,4)[above]{a};
        \node at (3.5,4)[above]{e};

        \node at (0,0)[below]{$w$};
        \node at (4,0)[below]{$u$};
        \node at (0,4)[above]{$v_0$};
        \node at (1,4)[below]{$v_0$};
        \node at (2,3)[below]{$v_3$};
        \node at (3,2)[below]{$v_2$};
        \node at (4,1)[right]{$v_1$};
        \node at (2,4)[above]{$w$};
        \node at (3,4)[above]{$u$};
        \node at (4,4)[above]{$w$};

        \draw[decoration={markings, mark=at position 0.5 with {\arrow{>}}},postaction={decorate}] (5,0) -- (6,0);
        \draw[decoration={markings, mark=at position 0.5 with {\arrow{>}}},postaction={decorate}] (5,4) -- (6,4);
        \draw[red] (5,0) -- (5,4);
        \draw[blue] (6,0) -- (6,4);
        \draw (6,0) -- (5,1);
        \draw (6,2) -- (5,3);
        \node at (5.5,0)[below]{b};
        \node at (5.5,4)[above]{a};
        \draw[decoration={markings, mark=at position 0.5 with {\arrow{>}}},postaction={decorate}] (7, 0) -- (9,0);
        \draw[decoration={markings, mark=at position 0.5 with {\arrow{>}}},postaction={decorate}] (7, 4) -- (8,4);
        \draw[decoration={markings, mark=at position 0.5 with {\arrow{>}}},postaction={decorate}] (8, 4) -- (9,4);
        \draw[blue] (7,0) -- (7,4);
        \draw[blue] (9,0) -- (9,4);
        \draw (9,0) -- (7,2);
        \draw (9,2) -- (8,3);
        \draw[red] (8,1) -- (8,4);
        \node at (8,0)[below]{c};
        \node at (7.5,4)[above]{e};
        \node at (8.5,4)[above]{a};

        \draw[decoration={markings, mark=at position 0.5 with {\arrow{>}}},postaction={decorate}] (10,0) -- (11,0);
        \draw[decoration={markings, mark=at position 0.5 with {\arrow{>}}},postaction={decorate}] (10,4) -- (11,4);
        \draw[blue] (10,0) -- (10,4);
        \draw[] (11,0) -- (11,4);
        \node at (10.5,0)[below]{d};
        \node at (10.5,4)[above]{b};

        \draw[decoration={markings, mark=at position 0.5 with {\arrow{>}}},postaction={decorate}] (12,0) -- (13,0);
        \draw[red] (12,0) -- (12,4);
        \draw[] (13,0) -- (13,4);
        \draw (12,3) -- (13,4);
        \draw[decoration={markings, mark=at position 0.5 with {\arrow{<}}},postaction={decorate}] (12,4) -- (13,4);
        \node at (12.5,0)[below]{e};
        \node at (12.5,4)[above]{d};

        \draw[violet,thick] (0,0.5) -- (2,0);
        \draw[violet,thick] (0,3) -- (4,0.5);
        \draw[violet,thick] (2.5,4) -- (3,3) -- (3.5,4);

        \draw[violet,thick] (5,0.5) -- (6,0.5);

        \draw[violet,thick] (5.5,4) -- (6,3);

        \draw[violet,thick] (7,0.5) -- (9,0.5);
        \draw[violet,thick] (7,3) -- (7.5,4);
        \draw[violet,thick] (8.5,4) -- (9,3);

        \draw[violet,thick] (10,0.5) -- (11,0.5);
        \draw[violet,thick] (10,3) -- (11,3);

        \draw[violet,thick] (12.5,0) -- (13,0.5);
        \draw[violet,thick] (12,0.5) -- (13,3);
    \end{tikzpicture}

    \caption{The folded mapping torus $X$ of $f$ with a cross section dual to the class $2r^*+b^*$. The horizontal edges form the base graph $\Gamma$ of \Cref{fig:train-track-representative}, and are not 1-cells of the trapezoidal cell structure of $X$. The vertical 1-cells are colored corresponding to the vertices of \Cref{fig:train-track-representative}. The top and bottom are identified by the edge labels. Skew 1-cells of the same height are identified, as are the 2-cells above them.}
    \label{fig:folded-mapping-torus}
\end{figure}
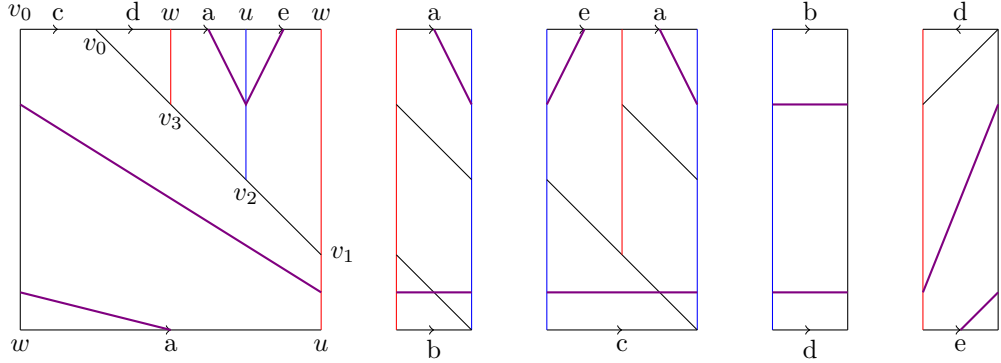

Given an expanding irreducible train track map $f: \Gamma\to \Gamma$ representing $\varphi$, one can define the mapping torus of $f$, the space $\Gamma \times I/(x,1)\sim (f(x),0)$. In \cite{DKL1}, given a Stallings fold decomposition of $f$, they construct the {\em folded mapping torus} of $f$, denoted $X = X_f$ (typically the dependence on the folding sequence is suppressed). The space $X$ is a quotient of the mapping torus of $f$ where at each height $t$, points of $\Gamma$ are identified according to a continuous version of the Stallings fold decomposition after time $t$. The space $X$ has the structure of a 2-complex with a {\em trapezoidal cell structure}. The {\em vertical 1-cells} of the structure lie above vertices of the graph $\Gamma$, sometimes after subdivisions. There is a {\em skew} (or {\em diagonal}) {\em 1-cell} for each fold in the Stallings fold sequence. The {\em degree} of a 1-cell $\bar{v}$ is the minimal number of components of $U\setminus\bar{v}$, where $U$ is an arbitrarily small neighborhood of a point $x\in \bar{v}$. Each skew 1-cell has degree three, is the bottom edge of a unique 2-cell, and is in a top edge of two 2-cells (or possibly a single 2-cell which folds onto itself). \Cref{fig:folded-mapping-torus,fig:folded-mapping-torus-cells} depict the folded mapping torus for the automorphism of \Cref{ex:automorphism}.

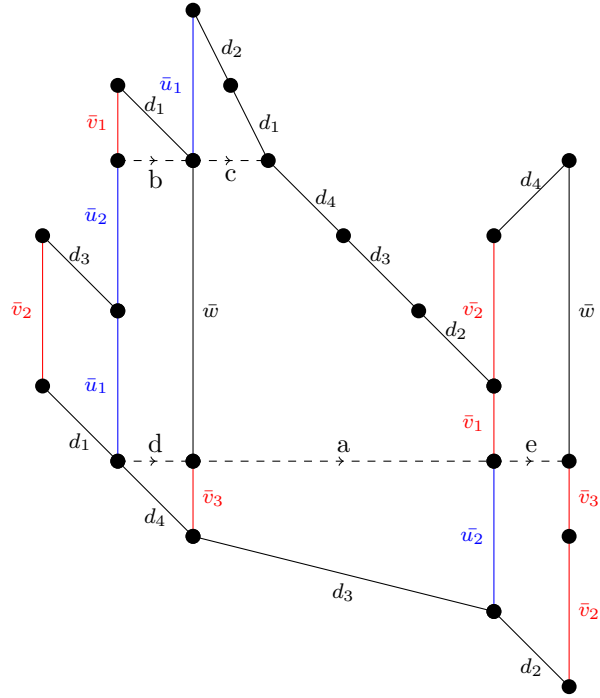
\begin{figure}
    \centering
    \begin{tikzpicture}

    \draw[red] (-1,1) -- (-1,3) node[midway,left]{\footnotesize $\bar{v}_2$};
    \draw (-1,1) -- (0,0) node[midway, below]{\footnotesize $d_1$};
    \draw (-1,3) -- (0,2) node[midway, above]{\footnotesize $d_3$};
        \node[circle,fill,inner sep=2pt] at (-1,1){};
        \node[circle,fill,inner sep=2pt] at (-1,3){};
        \node[circle,fill,inner sep=2pt] at (0,0){};
        \node[circle,fill,inner sep=2pt] at (0,2){};

    \draw[blue] (0,0) -- (0,4) node[pos=1/4,left]{\footnotesize $\bar{u}_1$} node[pos=13/16,left]{\footnotesize $\bar{u}_2$};
    \draw[red] (0,4) -- (0,5) node[midway,left]{\footnotesize $\bar{v}_1$};
    \draw[red] (1,-1) -- (1,0) node[midway,right]{\footnotesize $\bar{v}_3$};
    \draw (1,0) -- (1,4) node[midway,right]{\footnotesize $\bar{w}$};
    \draw (0,5) -- (1,4) node[midway, above]{\footnotesize $d_1$};
    \draw (0,0) -- (1,-1) node[midway, below]{\footnotesize $d_4$};
        \node[circle,fill,inner sep=2pt] at (0,0){};
        \node[circle,fill,inner sep=2pt] at (0,2){};
        \node[circle,fill,inner sep=2pt] at (0,4){};
        \node[circle,fill,inner sep=2pt] at (0,5){};
        \node[circle,fill,inner sep=2pt] at (1,-1){};
        \node[circle,fill,inner sep=2pt] at (1,0){};
        \node[circle,fill,inner sep=2pt] at (1,4){};

    \draw[blue] (1,4) -- (1,6) node[midway,left]{\footnotesize $\bar{u}_1$};
    \draw[blue] (5,-2) -- (5,0) node[midway,left]{\footnotesize $\bar{u_2}$};
    \draw (1,6) -- (2,4) node[pos=0.25, right]{\footnotesize $d_2$} node[pos=0.75,right]{\footnotesize $d_1$};
    \draw(2,4) -- (5,1) node[pos=1/6, right]{\footnotesize $d_4$} node[midway,above]{\footnotesize $d_3$} node[pos=5/6, above]{\footnotesize $d_2$};
    \draw (1,-1) -- (5,-2) node[midway,below]{\footnotesize $d_3$};
        \node[circle,fill,inner sep=2pt] at (1,-1){};
        \node[circle,fill,inner sep=2pt] at (1,0){};
        \node[circle,fill,inner sep=2pt] at (1,4){};
        \node[circle,fill,inner sep=2pt] at (1,6){};
        \node[circle,fill,inner sep=2pt] at (1.5,5){};
        \node[circle,fill,inner sep=2pt] at (2,4){};
        \node[circle,fill,inner sep=2pt] at (3,3){};
        \node[circle,fill,inner sep=2pt] at (4,2){};
        \node[circle,fill,inner sep=2pt] at (5,1){};
        \node[circle,fill,inner sep=2pt] at (5,0){};
        \node[circle,fill,inner sep=2pt] at (5,-2){};
    
    \draw[red] (5,0) -- (5,3) node[pos=1/6,left]{\footnotesize $\bar{v}_1$} node[pos=2/3,left]{\footnotesize $\bar{v_2}$};
    \draw[red] (6,-3) -- (6,0) node[pos=1/3,right]{\footnotesize $\bar{v}_2$} node[pos=5/6,right]{\footnotesize $\bar{v}_3$};
    \draw (6,0) -- (6,4) node[midway,right]{\footnotesize $\bar{w}$};
    \draw (5,3) -- (6,4) node[midway, above]{\footnotesize $d_4$};
    \draw(5,-2) -- (6,-3) node[midway, below]{\footnotesize $d_2$};
        \node[circle,fill,inner sep=2pt] at (5,-2){};
        \node[circle,fill,inner sep=2pt] at (5,0){};
        \node[circle,fill,inner sep=2pt] at (5,1){};
        \node[circle,fill,inner sep=2pt] at (5,3){};
        \node[circle,fill,inner sep=2pt] at (6,4){};
        \node[circle,fill,inner sep=2pt] at (6,0){};
        \node[circle,fill,inner sep=2pt] at (6,-1){};
        \node[circle,fill,inner sep=2pt] at (6,-3){};

    \draw[dashed,decoration={markings, mark=at position 0.5 with {\arrow{>}}}, postaction={decorate}] (0,0) -- (1,0) node[midway,above]{d};
    \draw[dashed,decoration={markings, mark=at position 0.5 with {\arrow{>}}}, postaction={decorate}] (0,4) -- (1,4) node[midway,below]{b};
    \draw[dashed,decoration={markings, mark=at position 0.5 with {\arrow{>}}},postaction={decorate}] (1,0) -- (5,0) node[midway,above]{a};
    \draw[dashed,decoration={markings, mark=at position 0.5 with {\arrow{>}}},postaction={decorate}] (5,0) -- (6,0) node[midway,above]{e};
    \draw[dashed,decoration={markings, mark=at position 0.5 with {\arrow{>}}},postaction={decorate}] (1,4) -- (2,4) node[midway,below]{c};
    \end{tikzpicture}
    
    \caption{The trapezoidal 2-cells of the folded mapping torus depicted in \Cref{fig:folded-mapping-torus}. The subscripts for $d_1, \dots, d_4$ on the skew 1-cells correspond to the height of the skew 1-cell in \Cref{fig:folded-mapping-torus}. Strictly speaking, both the `trapezoidal cell structure' of \cite{DKL1} and the modified version given here are further subdivisions of the cell structure shown above; we have also added additional vertices at the height of the graph $\Gamma$ for the sake of consistent colorings.}
    \label{fig:folded-mapping-torus-cells}
\end{figure}

The folded mapping torus $X$ is equipped with an upward semiflow $\psi$, which descends from a suspension semiflow on the mapping torus of $f$. To analogize the cones arising from Thurston's norm for 3-manifolds, an open cone $\mathcal{P}\subseteq H^1(X;\RR) = H^1(G; \RR)$ is defined in \cite{DKL1} as the set of cohomology classes which have cocycle represenatives that are positive on each 1-cell of $X$ (with the trapezoidal cell structure). Each primitive integral cohomology class $\alpha\in \mathcal{P}$ is dual to a cross section $\Theta_\alpha$ of the flow $\psi$ (in the sense of \cite{DKL1}), and the first return map $f_\alpha$ of $\psi$ is a homotopy equivalence of $\Theta_\alpha$. Thus each such $\alpha$ induces a splitting of $G_\varphi$ as $\pi_1(\Theta_\alpha)\rtimes \ZZ$ and is associated to an outer automorphism $[\varphi_\alpha: \pi_1(\Theta_\alpha)\to \pi_1(\Theta_\alpha)]$, which we call the {\em monodromy}.

A larger cone $C_X\subseteq H^1(X; \RR)$ investigated in \cite{DKL2} is also associated to a folding mapping torus $X$. Here, primitive integral classes $\alpha\in C_X$ are still dual to cross sections $\Theta_\alpha$, but the first return maps are not necessarily homotopy equivalences. These cross sections as first return maps induce a splitting of $G$ as an ascending HNN-extension of an injective endormorphism, and it is shown that the cone $C_X$ is equal to a component of $\RR_+\cdot \Sigma(G)$, where $\Sigma(G)$ is the BNS invariant introduced in \cite{BNS87}. We will consider a component of the cone over $\Sigma_s(G) = \Sigma(G) \cap -\Sigma(G)$ which we will call $C_s$: $C_s$ is the component of $C_X \cap (\RR_+\cdot \Sigma_s(G))$ that contains $\mathcal{P}$. Each integral class of $C_s$ induces a splitting of $G$ as a mapping torus group of an (outer) automorphism $\varphi_\alpha$, and so primitive integral classes in $\Sigma_s(G)$ have associated axis bundles.

\subsubsection{A refined cell structure for $X$.}
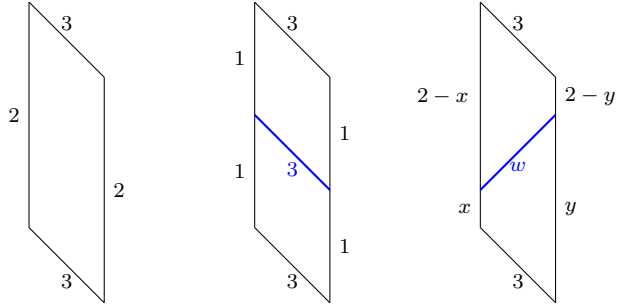
\begin{figure}
    \centering
    \begin{tikzpicture}
        \draw (-3,0) -- (-3,3) node[midway,left]{\footnotesize 2};
        \draw (-3,3) -- (-2,2) node[midway,above]{\footnotesize 3};
        \draw (-3,0) -- (-2,-1) node[midway,below]{\footnotesize 3};
        \draw (-2,-1) -- (-2,2) node[midway,right]{\footnotesize 2};
    
        \draw (0,0) -- (0,3) node[pos=1/4,left]{\footnotesize 1} node[pos=3/4,left]{\footnotesize 1};
        \draw (0,3) -- (1,2) node[midway,above]{\footnotesize 3};
        \draw (0,0) -- (1,-1) node[midway,below]{\footnotesize 3};
        \draw (1,-1) -- (1,2) node[pos=1/4,right]{\footnotesize 1} node[pos=3/4,right]{\footnotesize 1};
        \draw[blue,thick] (0,1.5) -- (1,0.5) node[midway,below]{\footnotesize 3};

        \draw (3,0) -- (3,3) node[pos=1/12,left]{\footnotesize $x$};
        \node at (2.5,1.75){\footnotesize $2-x$};
        \draw (3,3) -- (4,2) node[midway,above]{\footnotesize 3};
        \draw (3,0) -- (4,-1) node[midway,below]{\footnotesize 3};
        \draw (4,-1) -- (4,2) node[pos=5/12,right]{\footnotesize $y$} node[pos=11/12,right]{\footnotesize $2-y$};
        \draw[blue,thick] (3,0.5) -- (4,1.5) node[midway,below]{\footnotesize $w$};
    \end{tikzpicture}
    \caption{Depiction of a trapezoid and two possible subdivisions. The leftmost trapezoid has labels according to a positive cocycle $z$. The center trapezoid has been subdivided with a blue `parallel' skew 1-cell, and the cocycle $z$ has a positive refinement to the new cell structure. On the right, the additional blue skew 1-cell is nonparallel to the existing skew 1-cells. The cocycle condition for the rightmost trapezoid is that $3 + x + w = y$, which cannot be satisfied for positive $x,y,w$ without making another 1-cell have negative weight: if $x,w>0$, then $y > 3$, but then the weight on the top right 1-cell is $2-y < -1 < 0$. Geometrically, adding the skew 1-cell in a nonparallel manner has the effect of excluding cross sections with edges traveling between the two skew 1-cells.}
    \label{fig:subdividing-cocycles}
\end{figure}

    In the construction of the folded mapping torus \cite[Section 4.4]{DKL1}, the authors add some skew 1-cells in addition to those from the folding sequence. They do this to ensure that they can build a trapezoidal cell structure without horizontal 1-cells. For the general graph maps they consider, this is a necessary addition. For instance if an edge $e\in E(\Gamma)$ is a loop which is fixed by the graph map $f$ and not in the image of any other edge, it is necessary to add one of these additional skew 1-cells to obtain a trapezoidal cell structure on the folded mapping torus. But in the context of an expanding irreducible train track map, we will show this is never necessary in \Cref{lem:only-degree-3-skews}.

    The authors of \cite{DKL1} assert that the choice of additional skew 1-cell ``does not matter." However, we note that the choice of added 1-cell can change the cone $\mathcal{P}$: if $T$ is a trapezoid whose top and bottom sides are parallel, then $T$ can be subdivided by adding an additional skew 1-cell into two trapezoids which are not parallel. In this case, you cannot necessarily refine a positive cocycle to be positive on the subdivision, while you can if the subdivision was with a parallel skew 1-cell, see \Cref{fig:subdividing-cocycles}. One can also check that subdividing a nonparallel trapezoid will preserve $\mathcal{P}$ regardless of the direction.

\begin{lemma} \label{lem:auxillary-graph-no-cycles}
    Let $f: \Gamma\to \Gamma$ be a train track map and $G(f)$ be the combinatorial graph whose vertices are the set $E(\Gamma)$ with a single directed edge from $e_i$ to $e_j$ if $f(e_i)$ runs over $e_j$ exactly once and there is no $k\neq i$ such that $f(e_k)$ runs over $e_j$. If $f$ is expanding irreducible, $G(f)$ has no cycles.
\end{lemma}
\begin{proof}
    Suppose $G(f)$ has a cycle, say $e_1\mapsto e_2 \mapsto \dots \mapsto e_n\mapsto e_1$. If there exists an edge $e\in E(\Gamma)$ not in the cycle, then $f(e)$ does not run over any of the $e_i$, because each $e_i$ has an incoming edge not from $e$. Thus no power $f^k(e)$ runs over any of the $e_i$. In this case, $f$ is not irreducible.

    Suppose then that the cycle contains all edges of $\Gamma$, i.e. for each $e_i$ it maps over $e_{i+1}$ and no other edges. Then $f$ is a finite order graph map (in fact a homeomorphism) and so is not expanding.
\end{proof}

\begin{figure}
    \centering
    \begin{tikzpicture}
        \draw[very thick] (0,-1) -- (0,3);
        \draw[very thick] (1,-1.2) -- (1,9);
            \draw (2,-1.4) -- (2,0.5);
            \draw[very thick] (2,0.5) -- (2,9);
        \draw[very thick] (3,-1.6) -- (3,6);
            \draw (4,-1.8) -- (4,0.5);
            \draw[very thick] (4,0.5) -- (4,6);
        \draw[very thick] (5,-2) -- (5,6);
        \draw[very thick] (1,2) -- (0,3);
        \draw[very thick] (0,-1) -- (5,-2);
        \draw[dashed] (1,0) -- (3,1);
        \draw[dashed] (3,0) -- (5,1);
        
        \draw[very thick] (3,3) -- (2,4);
        \draw[dashed] (1,3) -- (2,4);

        \draw[very thick] (5,3) -- (4,4);
        \draw[very thick] (3,4) -- (4,5);
            \draw (3.2,4.2) -- (3.2,6);
            \draw[very thick] (3.4,4.4) -- (3.4,6);
            \draw (3.6, 4.6) -- (3.6,6);
            \draw[very thick] (3.8,4.8) -- (3.8,6);

        \draw[very thick] (2,8) -- (1,7);
            \draw (1.2,7.2) -- (1.2,9);
            \draw[very thick] (1.4,7.4) -- (1.4,9);
            \draw (1.6,7.6) -- (1.6,9);
            \draw[very thick] (1.8,7.8) -- (1.8,9);

        \draw[thick,blue,decoration={markings,mark=at position 0.5 with {\arrow{>}}},postaction=decorate] (0,0) -- (5,0) node[midway,below]{$a$};

        \draw[thick,blue,decoration={markings,mark=at position 0.5 with {\arrow{>}}},postaction=decorate] (0,3) -- (1,3) node[midway,above]{$e_1$};
        \draw[thick,blue,decoration={markings,mark=at position 0.5 with {\arrow{>}}},postaction=decorate] (1,3) -- (3,3) node[midway,above left]{$b$};
        \draw[thick,blue,decoration={markings,mark=at position 0.5 with {\arrow{>}}},postaction=decorate] (3,3) -- (5,3) node[midway,above left]{$c$};

        \draw[thick,blue,decoration={markings,mark=at position 0.5 with {\arrow{>}}},postaction=decorate] (1,6) -- (2,6) node[midway,above]{$d$};
        \draw[thick,blue,decoration={markings,mark=at position 0.5 with {\arrow{>}}},postaction=decorate] (2,6) -- (3,6) node[midway,above]{$e_2$};
        \draw[thick,blue,decoration={markings,mark=at position 0.5 with {\arrow{<}}},postaction=decorate] (3,6) -- (4,6) node[midway,above]{$a$};
        \draw[thick,blue,decoration={markings,mark=at position 0.5 with {\arrow{>}}},postaction=decorate] (4,6) -- (5,6) node[midway,above]{$e_2$};

         \draw[thick,blue,decoration={markings,mark=at position 0.5 with {\arrow{<}}},postaction=decorate] (1,9) -- (2,9) node[midway,above]{$a$};
    \end{tikzpicture}
    \caption{A piece of a folded mapping torus $X$ of a map $f: \Gamma\to\Gamma$. The image of the inclusion $\Gamma\subset G$ is shown in blue. The dashed skew 1-cells are the degree 2 skew 1-cells added in \cite{DKL1}, while the thick vertical and skew 1-cells are the other 1-cells of their cell structure. The thinner vertical segments are pieces of vertical 1-cells added so that we can remove the degree two skew 1-cells from the cell structure.}
    \label{fig:new-cell-structure}
\end{figure}

\begin{lemma} \label{lem:only-degree-3-skews}
    If $f: \Gamma\to \Gamma$ is an expanding irreducible train track map, then the folded mapping torus $X$ of $f$ has a trapezoidal cell structure without skew 1-cells of degree two.
\end{lemma}
\begin{proof}
    Let $\hat{X}$ be the folded mapping torus of $f$ with the trapezoidal cell structure of \cite{DKL1}. In this cell structure the vertical 1-cells are segments of flow lines that flow into a vertex of $\Gamma$, and there are skew 1-cells both along fold identifications and the additional degree two skew 1-cells we would like to remove.
    
    We claim that we can modify this cell structure as follows: remove the degree two skew 1-cells and extend each vertical 1-cell $\bar{v}$ backwards until it intersects a degree three skew 1-cell. Note that a `backwards flow' is defined everywhere except the degree three skew 1-cells, so we need to check that each vertical 1-cell does backwards flow into a degree three skew 1-cell. Let $e$ be an edge of $\Gamma\subset X$ such that for some point $x\in e$, $x$ flows into $\bar{v}$ (i.e. $\psi_t(x) \in \bar{v}$ for some $t > 0$). Let $s$ be a degree three skew 1-cell of $\hat{X}$, then $s$ forward flows into some edge $e_0\in E(\Gamma)$. Since $f$ is irreducible, there is a $k > 0$ such that $f^k(e_0)$ crosses over $e$. In particular, there is a point $y\in s$ that flows into $x$, thus into $\bar{v}$. Hence the vertical 1-cells do actually backwards flow into degree three skew 1-cells.

    We have shown that after we remove a degree two skew 1-cell, the bottom of each trapezoid is well-defined. We now address what could occur at the top. Since each skew 1-cell is potentially only a segment of the top of a trapezoid, we may not have a trapezoidal cell structure after removing the degree two skew 1-cells. We now claim that for each edge $e$ of $\Gamma$, it can be subdivided so that each subdivided edge flows into a degree three skew 1-cell of $\hat X$.

    Let $e\in E(\Gamma)$ and $\Gamma_0$ be the Stallings subdivision of $\Gamma$. We denote by $e_1,\dots, e_n\in E(\Gamma_0)$ the subdivided edges of $e$. For each $i$, we consider whether the strip above $e_i$ contains a skew 1-cell. The edge $e_i$ is labeled by an edge $a\in E(\Gamma)$. If another subdivided edge $b$ has the same label $a$, then $e_i$ and $b$ are folded together in the Stallings fold decomposition and there is a skew 1-cell in the strip above $e_i$. On the other hand, if $e_i$ is the unique subdivided edge with the label $a$, then there is a directed edge from $e$ to $a$ in the graph $G(f)$ of \Cref{lem:auxillary-graph-no-cycles}.
    If so, we subdivide $a$ by adding vertices at each point of $f\inv(V(\Gamma))$, which induces a subdivision of $e_i$ adding vertices at the points of $f^{-2}(V(\Gamma))$. Now, we investigate whether there are skew 1-cells in the strips above subdivided edges of $a$.
    As above, the only failure to contain a skew 1-cell in a strip is if there is a directed edge of $G(f)$ originating at $a$. The graph $G(f)$ is finite and has no cycles by \Cref{lem:auxillary-graph-no-cycles}, so this process of consecutively subdividing edges until each subdivided edge of $e$ flows into a skew 1-cell will terminate. We note that the subdivision given here actually agrees with the subdivision arising from the backwards flowing of the vertical 1-cells. See \Cref{fig:new-cell-structure}.
\end{proof}

\section{Discreteness of Lone Axis Automorphisms} \label{section:proof-main-theorem}

In this section we prove our main theorem, that lone axis automorphisms form a discrete set of $\Sigma(G)$. We will first prove some facts about the axis bundle.

\subsection{High dimensional axis bundles}
Consider an expanding irreducible train track map $f: \Gamma\to \Gamma$ where $\Gamma$ is equipped with the eigenmetric induced by $f$. For any illegal turn $\tau = \{e_i,e_j\}$ and small $\eps > 0$, we can define a folding path $\gamma_\tau: [0,\eps]\to CV_N$ as follows: $\gamma_\tau(t)$ is obtained by first identifying (folding) the initial length $t$ segments of the edges $e_i, e_j$, then rescaling all edge lengths so that their sum is one. The following lemma follows directly from the fold lines characterization of the axis bundle given in \cite{hm11axis_bundle}. The equivalence given by Handel--Mosher of that definition and the train track based definition we use goes through a more technical third definition of the axis bundle (`weak' train tracks), so we will give a proof here.

\begin{lemma} \label{lem:folding-stay-axis-bundle}
    Let $f: \Gamma\to \Gamma$ be an expanding irreducible train track map representing $\varphi$ and suppose $\Gamma$ is equipped with the eigenmetric induced by $f$. Then for any illegal turn $\tau$, $\gamma_\tau([0,\eps])\subseteq \mathcal{A}_\varphi$ for small $\eps > 0$.
\end{lemma}
\begin{proof}
    Let $\tau = \{e_i, e_j\}$ be an illegal turn at the vertex $v$ and $n$ be the minimal positive integer such that $\{Df^n(e_{i}),Df^n(e_j) \}$ is the degenerate turn $\{a,a\}$. In a Stallings fold decomposition of $f^n$, we can take our first fold to be between subdivided edges $e_{i_1},e_{j_1}$ of $e_i,e_j$, respectively, because they both carry the edge label $a$. The length of $e_{i_1}$ and $e_{j_1}$ is $\ell(a)/\lambda^n$, where $\lambda$ is the stretch factor of $f$, so we can fold lengths up to $\ell(a)/\lambda^n$. We choose $0< \eps < \ell(a)/\lambda^n$ and claim that $\gamma_\tau([0,\eps])\subseteq \overline{TT(\varphi^n)}$.

    Consider the set $\mathcal{V} = \bigcup_{k=1}^\infty f^{-nk}(V(\Gamma))$. Since $f$ is expanding, $\mathcal{V}$ is dense in $\Gamma$. The edge $e_i$ is the isometric image of a map $\iota: [0,\ell_\Gamma(e_i)]\to \Gamma$, with $\iota(0) = v\in V(\Gamma)$. The set $\iota\inv(\mathcal{V})$ is dense in $[0,\eps]$. Suppose $t\in \iota\inv (\mathcal{V})\cap [0,\eps]$, then $\iota(t)\in f^{-nk}(V(\Gamma))$ for some $k$. We can subdivide the edges of $\Gamma$ at every point in $f^{-nk}(V(\Gamma))$, then perform the fold of the subdivided $e_i$ and $e_j$ of length $t$. The train track map $f^n$ of $\Gamma$ descends to a train track map of the folded graph $\gamma_\tau(t)$, so $\gamma_\tau(t)\in TT(\varphi^n)$.
    
    For $t\notin \iota\inv(\mathcal{V})$, we note that $\gamma_\tau$ is a continuous function (it continuously changes the edge lengths). So density of $\iota\inv(\mathcal{V})$ in $[0,\eps]$ implies that $TT(\varphi^n)$ intersects $\gamma_\tau([0,\eps])$ in a dense set, i.e. $\gamma_\tau([0,\eps])\subseteq \overline{TT(\varphi^n)}\subseteq \mathcal{A}_\varphi$.
\end{proof}

It will be useful for us to be able to perform simultaneous folds while remaining in the axis bundle. If $\tau_1, \dots, \tau_n$ are $n$ illegal turns at distinct vertices $v_1, \dots, v_n$ and $\eps > 0$ is sufficiently small, we define a map $\Delta: [0,\eps]^n\to CV_N$ as follows: $\Delta(t_1, \dots, t_n)$ is obtained by (simultaneously) folding at each turn $\tau_i$ by a length $t_i$, then rescaling the metric to be of volume one.

\begin{corollary} \label{cor:delta-map-folding-stay-axis-bundle}
    Let $f: \Gamma\to\Gamma$ be an expanding irreducible train track map representing $\varphi$ and suppose $\Gamma$ is equipped with the eigenmetric induced by $f$. Then for any collection of illegal turns $\tau_1, \dots, \tau_m$ at distinct vertices $v_1,\dots, v_m$, $\Delta([0,\eps]^m)\subseteq \mathcal{A}_\varphi$ for small $\eps > 0$.
\end{corollary}
\begin{proof}
    We write $\tau_i = \{e_{i_1},e_{i_2}\}$ and let $n$ be the least integer such that the turn $\{Df^n(e_{i_1}),Df^n(e_{i_2})\}$ is degenerate for all $i$, calling this degenerate turn $\{a_i,a_i\}$. We choose $0 < \eps < \min_{i=\{1,\dots,n\}} (\ell(a_i)/\lambda^n)$. Again we let $\mathcal{V} = \bigcup_{k=1}^\infty f^{-nk}(V(\Gamma))$ and say $\iota_i: [0,\ell_\Gamma(e_{i_1})]\to\Gamma$ is the isometric embedding onto $e_{i_1}$. Consider $\Delta(t_1, \dots, t_m)$. If $t_i\in \iota_i\inv (\mathcal{V})$ for each $i$, then we can consecutively apply \Cref{lem:folding-stay-axis-bundle} on each illegal turn $\tau_i$. After each fold, we remain in $TT(\varphi^n)$, so that eventually we obtain that $\Delta(t_1,\dots, t_m)\in TT(\varphi^n)$. A density and continuity argument similar to that of Lemma 3.1 yields that $\Delta(t_1, \dots, t_m)\in \mathcal{A}_\varphi$ for any $(t_1, \dots, t_m)$.
\end{proof}

We now work towards proving conditions under which an automorphism $\varphi$ has many axes. The following lemma is implied by the proof of Lemma 4.3 of \cite{mp16lone-axes}.

\begin{lemma}[\cite{mp16lone-axes}, Lemma 4.3] \label{lem:mp16-lone-axis-unique-illegal-turn}
    Let $\varphi: F_N\to F_N$ be a nongeometric automorphism. If $\varphi$ has a train track representative $f: \Gamma\to\Gamma$ with at least two illegal turns, $\mathcal{A}_\varphi$ is not a lone axis.
\end{lemma}

\begin{proof}
    Write $\tau_i = \{e_{i_1},e_{i_2}\}$ and let $n$ be the minimal positive integer such that the turns $\{Df^n (e_{i_1}), Df^n (e_{i_2})\}$ are degenerate. Then $f^n$ has two distinct Stallings fold decompositions, one for each choice of $\tau_i$ to fold first. These two choices give rise to two distinct periodic fold lines in $\mathcal{A}_\varphi$, so $\varphi$ has multiple axes.
\end{proof}

Note that the statement of \Cref{lem:mp16-lone-axis-unique-illegal-turn} in \cite{mp16lone-axes} assumes that $\varphi$ is ageometric rather than nongeometric, because this is a small part of a larger proof. On the other hand, as already noted, it is shown in \cite{hm11axis_bundle} that parageometric automorphisms do not ever have a lone axis. We now prove an analogous lemma, that if a train-track map $f: \Gamma\to \Gamma$ has many illegal turns, it is far from being lone axis in the sense that $\mathcal{A}_\varphi$ has large dimension. Note that as stated, it doesn't recover \Cref{lem:mp16-lone-axis-unique-illegal-turn}.

\begin{lemma} \label{lem:high-local-dimension}
    Let $f: \Gamma \to \Gamma$ be an expanding irreducible train track map representing $\varphi: F_N\to F_N$ satisfying the following properties:
    \begin{enumerate}
        \item $\Gamma$ has $n$ illegal turns $\tau_1, \dots, \tau_n$, occuring at distinct valence three vertices $v_1, \dots, v_n$.
        \item There is some edge $e_0$ of $\Gamma$ not incident to any $v_i$.
    \end{enumerate}
    Then there is an embedded copy of $\mathbb{R}^n$ in any neighborhood of $\Gamma$ which is contained in the axis bundle $\mathcal{A}_\varphi$. In particular, $\mathcal{A}_\varphi$ has local dimension of at least $n$ at $\Gamma$.
    \par Additionally, if we remove assumption (2) above, then $\mathcal{A}_\varphi$ has local dimension of at least $n-2$ at $\Gamma$.
\end{lemma}
\begin{proof}
    Consider the simultaneous folding map $\Delta: [0,\eps]^n\to CV_N$, where $\eps > 0$ is chosen small enough so that \Cref{cor:delta-map-folding-stay-axis-bundle} guarantees $\Delta([0,\eps]^n)\subseteq \mathcal{A}_\varphi$. Note that because each illegal turn is at a valence three vertex, performing a small fold at a turn doesn't change the homeomorphism type of the graph, so we may refer to edges of $\ell_{\Delta(t_1,\dots,t_n)}$ by the names of edges of $\Gamma$. For any edge $e\in\Gamma$ we have that
    \[
        \ell_{\Delta(t_1,\dots,t_n)}(e) = \frac{\ell_\Gamma(e) + \sum_{i=1}^n \delta_i(e) t_i}{1-\sum_{i=1}^n t_i}
    \]
    where
    \[
        \delta_i(e) = \begin{cases}
            0 & \text{e is not incident to $v_i$ or $e$ is a loop} \\
            1 & \text{$e$ is incident to $v_i$ but not part of the illegal turn $\tau_i$} \\
            -1 & \text{$e$ is part of the illegal turn $\tau_i$.}
        \end{cases}
    \]
    Now we show that $\Delta$ is injective. Suppose that $\Delta(t_1,\dots, t_n) = \Delta(t_1', \dots, t_n')$. For each edge $e\in E(\Gamma)$, we obtain the following by equating $\ell_{\Delta(t_1,\dots,t_n)}(e) = \ell_{\Delta(t_1',\dots,t_n')}(e)$:
    \[
        \frac{1}{1-\sum_{i=1}^n t_i} \left ( \ell_\Gamma(e) + \sum_{i=1}^n \delta_i(e) t_i \right ) = \frac{1}{1-\sum_{i=1}^n t_i'} \left ( \ell_\Gamma(e) + \sum_{i=1}^n \delta_i(e) t_i' \right ).
    \]
    By (2), there is some $e_0$ such that $\delta_i(e_0) = 0$ for all $i$, which implies that $\sum_{i=1}^n t_i = \sum_{i=1}^n t_i'$. Let $s_i = t_i - t_i'$, the system of equations becomes $\sum_{i=1}^n \delta_i(e) s_i = 0$ over all edges $e$ of $\Gamma$.

    Since $\Gamma$ is connected and by assumption (2), there is an edge $e$ between one of the $v_i$ and a vertex $w\neq v_j$. Without loss of generality assume $i=1$. Then the equation for $e$ has the form $\delta_1(e) s_1 = 0$, and $\delta_1(e) = \pm 1$. So $s_1 = 0$. Now we claim that if $s_i = 0$ and there is an edge $e_{ij}$ between the vertices $v_i$ and $v_j$, then $s_j = 0$. This again comes from considering the equation associated to $e_{ij}$: $\delta_i(e_{ij}) s_i + \delta_j (e_{ij}) s_j = 0$. Since $s_i = 0$ and $\delta_j(e_{ij}) \neq 0$ the claim follows. This shows that if there is an edge path between $v_j$ and $w$, then $s_j = 0$, and this is true because $\Gamma$ is connected. Thus $t_i = t_i'$ for all $i$ and $\Delta$ is injective. Now $\Delta$ is a closed map since its domain $[0,\eps]^n$ is compact, so $\Delta$ is an embedding.

    Removing assumption (2), we let $w$ and $w'$ be any two vertices of $\Gamma$ connected by an edge $e_0$ and apply the above argument without considering folds at $w$ or $w'$. That is, after removing the possible folds at $w$ or $w'$, we recover assumption (2) replacing $n$ with $n-2$.
\end{proof}

Pfaff and Tsang \cite{pfaff-tsang-cubist} have given the axis bundle a cell structure that makes $\mathcal{A}_\varphi$ into what they call a `cubist' complex. The embedding we describe in \Cref{lem:high-local-dimension} lies in a face of a `branched cube' in this complex, and in fact \Cref{lem:high-local-dimension} can be viewed as a consequence of \cite[Lemma 3.8]{pfaff-tsang-cubist}.

\subsection{Cross Sections With Many Illegal Turns}
We now begin to consider cross sections of the semiflow on the folded mapping torus.

Let $\varphi: F_N\to F_N$ be a nongeometric, fully irreducible automorphism and $G = F_N\rtimes_\varphi \ZZ$ the associated free-by-cyclic group. We fix some folding sequence of a train track representative $f: \Gamma\to \Gamma$ and use this to define a folded mapping torus $X$ of $f$ with upward semiflow $\psi$. Let $\mathcal{P}$ denote the positive cone of \cite{DKL1} associated to $X$.

Given an integral class $\alpha\in \mathcal{P}$ and a positive cocycle $z$ representing $\alpha$, we obtain a `fibration' map $\eta_z: X\to S^1$ à la \cite{DKL1}. Then, for an appropriate $y\in S^1$, the set $\Theta_z = \eta_z\inv(y)$ is a graph which is a cross section to the semiflow $\psi$, and $\Theta_z$ has a vertex structure so that the first return map of $\psi$ is an expanding irreducible train track map. The graph $\Theta_z$ is connected if and only if $\alpha$ is primitive integral \cite[Theorem A]{DKL1}. The number of intersections of $\Theta_z$ with a 1-cell $e$ is bounded below by $\lfloor z(e) \rfloor$, because $z(e)$ measures how much $e$ wraps around $S^1$ under $\eta_z$.

\begin{figure}
    \centering
\begin{tikzpicture}
    \draw (0,-2) -- (0,2);
    \draw (2,-2) -- (2,2);
    \draw (0,0) -- (2,-1);

    \draw (-1,-2) -- (0,0);
    \draw[densely dotted] (1.5,-2) -- (2,-1);

    \draw[violet,thick] (0,-1) -- (2,0);
    \draw[violet,thick] (0,0.2) -- (2,1.2);
    \node at (1,0.7) [draw, shape = circle, fill = black, minimum size = 0.1cm, inner sep = 0pt,violet] (){};
    \draw[violet,densely dotted,thick] (0,-0.5) -- (1,-0.5);
    \draw[violet, thick] (0,-0.5) -- (-0.25,-0.5);
\end{tikzpicture}
\caption{Each intersection of a section graph $\Theta$ (in purple) with a skew 1-cell of $X$ gives rise to an illegal turn in the first return map of $\psi$.}
\label{fig:flow-illegal-turn}
\end{figure}
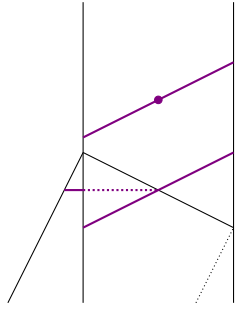

\begin{lemma}
    \label{two-skew-intersections-illegal}
    Let $\Theta_z$ be a section of $\psi$ corresponding to a positive cocycle $z$. For each intersection of $\Theta_z$ with a skew 1-cell $d$, there is an illegal turn in the first return map corresponding to the edges of $\Theta_z$ below the skew 1-cell (below meaning with respect to the semiflow $\psi$).
\end{lemma}
\begin{proof}
    Let $x$ denote a point of intersection of $\Theta_z$ and $d$.
    For any small $\eps > 0$, the semiflow $\psi_\eps$ at time $\eps$ will identify points on edges below the skew 1-cell $d$, so the corresponding directions of $\Theta_z$ map to the same direction of $\psi_\eps(\Theta_z)$.
    The result follows because the first return map factors through $\psi_\eps$. See \Cref{fig:flow-illegal-turn}.
\end{proof}
\begin{theorem} \label{thm:projectively-discrete}
    Let $\varphi$ be a nongeometric fully irreducible automorphism of $F_N$. Let $G = \langle w, r | w\in F_N, r\inv w r = \varphi(w) \rangle$ be the mapping torus group of $\varphi$; denote by $r^*$ the class in $H^1(G;\RR)$ with associated monodromy $\varphi$.
    
    Given $k > 0$, there is an open cone $C \subseteq C_s \subseteq \mathbb{R}_+\cdot \Sigma_s(G)\subseteq H^1(G;\RR)$ containing $r^*$ such that for any primitive integral class $\beta\in C\setminus (\RR_+\cdot \{r^*\})$, the dimension of the axis bundle $\mathcal{A}_{\varphi_\beta}$corresponding to $\beta$ is at least $k$. In other words, automorphisms with any given bound on the dimension of their axis bundles are discrete in $\Sigma_s(G)$.
\end{theorem}
\begin{proof}
Let $X$ be a folded mapping torus of a train track map representing $\varphi$, and let $\alpha_1,\dots,\alpha_{b-1}\in H^1(G;\ZZ)\subset H^1(G;\RR)$ be such that $\alpha_1,\dots,\alpha_{b-1},r^*$ form a basis of $H^1(G;\ZZ)$. We pick a positive cocycle $z_r$ representing $r^*$ and for each $\alpha_i$, we choose a 1-cocycle $z_i$ representing $\alpha_i$.

We can now define the cone $C$ in terms of these cocycles. Pick $M_0$ large enough so that for all $1$-cells $e$ in the trapezoidal cell structure of the folded mapping torus, $M_0 z_r(e) - \sum_{i=1}^{b-1} |z_i(e)| > 0$. Let $d$ be a skew 1-cell, and pick $M_d$ to be large enough so that $M_d z_r(d) - \sum_{i=1}^{b-1} |z_i(d)| > k + 2$. Let $M$ be an integer greater than $M_0$ and $M_d$. Define $C\subseteq H^1(G;\RR)$ to be the cone over the set $\{Mr^* + \sum_{i=1}^{b-1} \eps_i \alpha_i \ | \ \eps_i\in (-1,1)\}$. Any $\beta = c_r r^* + \sum_{i=1}^{b-1} c_i \alpha_i\in C$ is represented by a cocycle $z_\beta = c_r z_r + \sum_{i=1}^{b-1} c_i z_i$. This cocycle is positive by the choice of $M > M_0$ and by the triangle inequality.

\par We claim that for any primitive integral class $\beta\in C\setminus \{r^*\}$, the axis bundle $\mathcal{A}_{\varphi_\beta}$ of the monodromy $\varphi_\beta$ has dimension at least $k$. We fix the positive cocycle $z_\beta$ representing $\beta$ described above. Note that since $\beta$ is primitive integral and not on the ray through $r^*$, we have that $c_r > M$. Additionally, $c_r\geq M|c_i|$ for all $1\leq i \leq b-1$ by definition of $C$. Then 
\begin{align*}
    z_\beta(d) & = \frac{c_r}{M} \left ( Mz_r(d) + \sum_{i=1}^{b-1} \frac{Mc_i}{c_r} z_i(d) \right ) \\
    & \geq \frac{c_r}{M} \left ( Mz_r(d) - \sum_{i=1}^{b-1} \left | \frac{Mc_i}{c_r} z_i(d)\right | \right ) \\
    & \geq \frac{c_r}{M} \left ( Mz_r(d) - \sum_{i=1}^{b-1} |z_i(d)| \right ) > k + 2.
\end{align*}
Thus $\Theta_{z_\beta}$ has at least $k+2$ points of intersection with the skew 1-cell $d$, hence the first return map has at least $k+2$ illegal turns at vertices of valence three by \Cref{two-skew-intersections-illegal}. Since $\beta$ is a primitive integral class in $\mathcal{P}$, $\Theta_{z_\beta}$ is connected, and the first return map of $\psi$ defines an outer automorphism $[\varphi_\beta]$ of $\pi_1(\Theta_{z_\beta})$. By \Cref{lem:high-local-dimension} the local dimension of the axis bundle $\mathcal{A}_{\varphi_\beta}$ at $\Theta_{z_\beta}$ is greater than or equal to $(k+2)-2=k$. This proves our result.
\end{proof}

We note that \Cref{thm:main-theorem-intro-version} is a consequence of \Cref{thm:projectively-discrete}: if $\varphi$ has a lone axis, then $\mathcal{A}_\varphi$ is homeomorphic to $\RR$, which is one-dimensional.

\section{Extended Example} \label{section:extended-example}
In this section we further investigate the automorphism $\varphi$ of \Cref{ex:automorphism}. Studying this automorphism is what led the author to \Cref{thm:projectively-discrete}. We show that the discreteness conclusion of \Cref{thm:projectively-discrete} cannot be improved to a finiteness statement. Specifically, for the free-by-cyclic group specified by $\varphi$ from \Cref{ex:automorphism}, we find that there are infinitely many primitive integral classes whose mondromies are lone axis. In fact, lone axis monodromies occur at every class that satisfies a necessary cohomological condition.

\begin{theorem} \label{thm:extended-example}
    Let $\varphi: F_3\to F_3$ be the automorphism defined by
    \begin{align*}
        a & \mapsto ca \\
        b & \mapsto ab \\
        c & \mapsto b\inv ab
    \end{align*}
    and $G = F_3\rtimes_\varphi \ZZ$ the mapping torus group of $\varphi$ with stable letter $r$. In the component $C_s$ of $\Sigma_s(G)$ determined by $\varphi$, there are infinitely many primitive integral classes with lone axis monodromies. Further, each such class lies on the intersection of $C_s$ with a codimension one affine subspace of $H^1(G;\RR)$.
\end{theorem}

As discussed in \Cref{ex:automorphism,ex:lone-axis}, $\varphi$ is an ageometric fully irreducible automorphism with a lone axis in outer space, \Cref{fig:train-track-representative} shows a train track representative, and \Cref{fig:folded-mapping-torus,fig:folded-mapping-torus-cells} show the folded mapping torus $X$. We note that since $\varphi$ has a lone axis, any train track representative of $\varphi$ must lie on this axis, and so $X$ is unique up to flow line preserving homeomorphism.

The mapping torus group $G$ has the presentation $\langle a,b,c,r \ | \ r\inv a r = ca, r\inv b r = ab, r\inv c r = b\inv a b \rangle$. In the abelianization, we have that $c = 0$ from the first relation, then $a = 0$ from the third, so $H_1(G;\RR)\cong \RR^2$ and $b,r$ form a basis. Using the labels from \Cref{fig:folded-mapping-torus-cells}, we claim that $b,r$ are represented by the 1-cycles $\bar w - \bar v_2 + \bar v_3 - d_1 - 2d_4$ and $\bar v_1 + d_2 + \bar u_2$, respectively. We define $b^*,r^*\in H^1(G;\RR)$ to be the dual basis of $b,r$. As our first step in the proof, we show that $\mathcal{P} = C_s$.

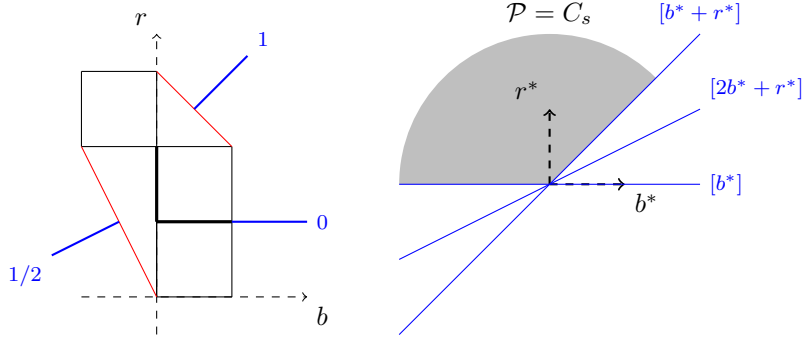
\begin{figure}
    \centering
\begin{tikzpicture}
    \draw[dashed, ->] (-1,0) -- (2,0) node[anchor=north west] {$b$};
    \draw[dashed, ->] (0,-0.5) -- (0,3.5) node[anchor = south east] {$r$};

    \draw (0,0) -- (0,1);
    \draw[very thick] (0,1) -- (0,2);
    \draw (0,0) -- (0,3);
    \draw (0,3) -- (-1,3);
    \draw (-1,3) -- (-1,2);
    \draw (-1,2) -- (0,2);
    \draw[very thick] (0,1) -- (1,1);
    \draw (1,1) -- (1,2);
    \draw (1,2) -- (0,2);
    \draw (0,2) -- (0,1);
    \draw (1,1) -- (1,0);
    \draw (1,0) -- (0,0);

    \draw[red] (0,0) -- (-1,2);
    \draw[red] (1,2) -- (0,3);

    \draw[blue,thick] (1,1) -- (2,1) node[anchor=west]{\footnotesize 0};
    \draw[blue,thick] (-0.5,1) -- + (206.565:1cm) node[anchor=north east]{\footnotesize 1/2}; 
    \draw[blue,thick] (0.5,2.5) -- + (45:1cm) node[anchor = south west]{\footnotesize 1}; 
    
\end{tikzpicture}
\qquad
\begin{tikzpicture}
    \fill[lightgray] (0,0) -- (-2,0) arc[start angle = 180, end angle = 45, radius = 2 cm] -- (0,0);
    \node[anchor=south] at (0,2){$\mathcal{P} = C_s$};
    \draw[dashed, thick, ->] (0,0) -- (1,0) node[anchor=north west] {$b^*$};
    \draw[dashed, thick, ->] (0,0) -- (0,1) node[anchor = south east] {$r^*$};
    \draw[blue] (0,0) -- (-2,-2);
    \draw[blue] (0,0) -- (2,2) node[anchor=south]{\footnotesize $[b^*+r^*]$};
    \draw[blue] (0,0) -- (2,1) node[anchor=south west]{\footnotesize $[2b^*+r^*]$};
    \draw[blue] (0,0) -- (-2,-1);
    \draw[blue] (0,0) -- (-2,0);
    \draw[blue] (0,0) -- (2,0) node[anchor=west]{\footnotesize $[b^*]$};
\end{tikzpicture}
    
    \caption{Brown Algorithm for calculation of BNS invariant for our example $G_\varphi$ discussed in \Cref{lem:brown-alg-equality}. The left picture shows the tracing out of the polygon from the algorithm, which reads off the relation of the 2-generator group ($r^3 b\inv r\inv b r\inv brb\inv r\inv br\inv b\inv$); the thick black edges are the edges traversed multiple times.
    The red edges are the edges of the polygon coming from taking the convex hull of the edges traced by the relator. The blue line segments perpendicular to some of the edges correspond to directions not in $\Sigma_s(G_\varphi)$, with their slopes labeled.
    \newline The right picture shows $H^1(G;\RR)$, and the blue lines here correspond to directions not in $\Sigma_s(G)$. The shaded gray region is $C_s$, a component of $\RR_+\cdot \Sigma_s(G)$. By \Cref{lem:brown-alg-equality}, this is equal to the positive cone $\mathcal{P}$.}
    \label{fig:brown-algorithm}
\end{figure}

\begin{lemma} \label{lem:brown-alg-equality}
    For the $\varphi$ of \Cref{thm:extended-example}, the positive cone $\mathcal{P}$ associated to $X$ is equal to $C_s$, which is a component of $\RR_+\cdot \Sigma_s(G)$.
\end{lemma}

\begin{proof}
    We first use Tietze transformations to find a two-generator one-relator presentation of $G$:
    \begin{align*}
        G & = \langle a,b,c,r | r\inv a r = ca, r\inv b r = ab, r\inv c r = b\inv a b \rangle \\
        & \cong \langle b,r | r^3 b\inv r\inv b r\inv brb\inv r\inv br\inv b\inv \rangle.
    \end{align*}
    There is an algorithm to calculate $\Sigma(G)$ given this two-generator one-relator presentation due to Brown \cite[Section 4]{Brown1987}, though we note that Brown's $\Sigma(G)$ and ours disagree by a negative due to a difference in convention between left and right actions, see \cite{DKL2}.
    We now describe Brown's procedure in our case; see \Cref{fig:brown-algorithm}. In the $(b,r)$-plane, we trace out an associated polygon by reading off the relator of the presentation: we draw a length one edge in the positive (negative) $b$ direction for each positive (negative) power of $b$, and we similarly travel in the $r$ direction for each power of $r$. We take the convex hull $P$ of this polygon, and we consider the directions perpendicular to some of the edges of $P$. The outward perpendicular of any diagonal edge of $P$ is not in $-\Sigma(G)$ (Brown's $\Sigma(G)$), nor are outward perpendiculars of any vertical or horizontal edge that has length at least two. Other directions can potentially be excluded if any `corner vertex' of $P$ was traversed twice while reading off the relator; this does not occur in this example. Thus $\Sigma_s(G)$ is the set of directions except $\pm (b^* + r^*), \pm (2b^*+r^*), \pm b^*$. So the component $C_s$ of $\RR_+\cdot \Sigma_s(G)$ containing $r^*$ is the cone over the set $\{t b^* + r^* \ | \ -\infty < t < 1\}$, since the slopes $0$ and $1$ represent the boundary of the component.

    In \Cref{fig:positive-cocycles} we construct a family of positive cocycles $z_t$ for $t < 1$. We claim that $z_t$ represents the class $r^* + tb^*\in C_s$. Noting that $\delta = 1-t-3\eps$, we have that
    \begin{align*}
        z_t(b) & = z_t(\bar w - \bar v_2 + \bar v_3 - d_1 - 2d_4) = 1 - 2\eps + \delta - \eps - 2\delta = t; \\
        z_t(r) & = z_t(\bar v_1 + d_2 + \bar u_2) = \eps + \eps + (1-2\eps) = 1.
    \end{align*}
    Since we have a positive cocycle representing $r^* + tb^*$ for all $t < 1$, this implies that $\{tb^* + r^* \ | \ -\infty < t < 1\}\subseteq \mathcal{P}$, completing the proof.
\end{proof}

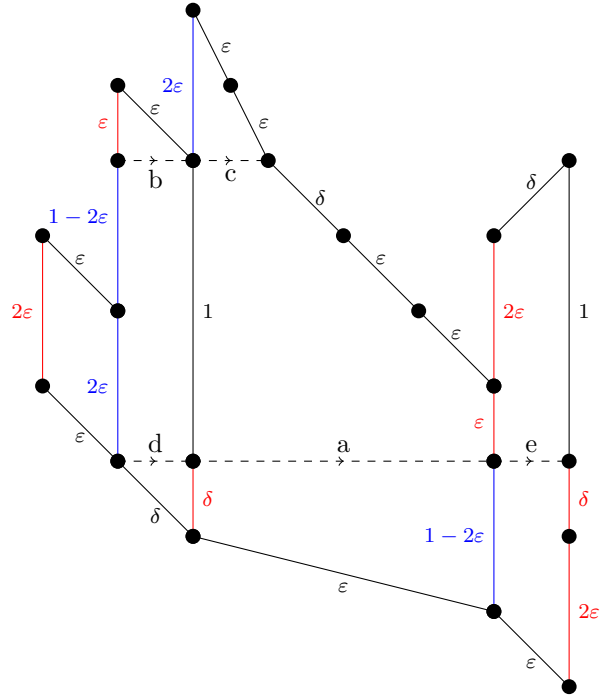
\begin{figure}
    \centering
    \begin{tikzpicture}
    \draw[red] (-1,1) -- (-1,3) node[midway,left]{\footnotesize $2\varepsilon$};
    \draw (-1,1) -- (0,0) node[midway, below]{\footnotesize $\varepsilon$};
    \draw (-1,3) -- (0,2) node[midway, above]{\footnotesize $\varepsilon$};
        \node[circle,fill,inner sep=2pt] at (-1,1){};
        \node[circle,fill,inner sep=2pt] at (-1,3){};
        \node[circle,fill,inner sep=2pt] at (0,0){};
        \node[circle,fill,inner sep=2pt] at (0,2){};

    \draw[blue] (0,0) -- (0,4) node[pos=1/4,left]{\footnotesize $2\varepsilon$} node[pos=13/16,left]{\footnotesize $1-2\varepsilon$};
    \draw[red] (0,4) -- (0,5) node[midway,left]{\footnotesize $\varepsilon$};
    \draw[red] (1,-1) -- (1,0) node[midway,right]{\footnotesize $\delta$};
    \draw (1,0) -- (1,4) node[midway,right]{\footnotesize $1$};
    \draw (0,5) -- (1,4) node[midway, above]{\footnotesize $\varepsilon$};
    \draw (0,0) -- (1,-1) node[midway, below]{\footnotesize $\delta$};
        \node[circle,fill,inner sep=2pt] at (0,0){};
        \node[circle,fill,inner sep=2pt] at (0,2){};
        \node[circle,fill,inner sep=2pt] at (0,4){};
        \node[circle,fill,inner sep=2pt] at (0,5){};
        \node[circle,fill,inner sep=2pt] at (1,-1){};
        \node[circle,fill,inner sep=2pt] at (1,0){};
        \node[circle,fill,inner sep=2pt] at (1,4){};

    \draw[blue] (1,4) -- (1,6) node[midway,left]{\footnotesize $2\varepsilon$};
    \draw[blue] (5,-2) -- (5,0) node[midway,left]{\footnotesize $1-2\varepsilon$};
    \draw (1,6) -- (2,4) node[pos=0.25, right]{\footnotesize $\varepsilon$} node[pos=0.75,right]{\footnotesize $\varepsilon$};
    \draw(2,4) -- (5,1) node[pos=1/6, right]{\footnotesize $\delta$} node[midway,above]{\footnotesize $\varepsilon$} node[pos=5/6, above]{\footnotesize $\varepsilon$};
    \draw (1,-1) -- (5,-2) node[midway,below]{\footnotesize $\varepsilon$};
        \node[circle,fill,inner sep=2pt] at (1,-1){};
        \node[circle,fill,inner sep=2pt] at (1,0){};
        \node[circle,fill,inner sep=2pt] at (1,4){};
        \node[circle,fill,inner sep=2pt] at (1,6){};
        \node[circle,fill,inner sep=2pt] at (1.5,5){};
        \node[circle,fill,inner sep=2pt] at (2,4){};
        \node[circle,fill,inner sep=2pt] at (3,3){};
        \node[circle,fill,inner sep=2pt] at (4,2){};
        \node[circle,fill,inner sep=2pt] at (5,1){};
        \node[circle,fill,inner sep=2pt] at (5,0){};
        \node[circle,fill,inner sep=2pt] at (5,-2){};
    
    \draw[red] (5,0) -- (5,3) node[pos=1/6,left]{\footnotesize $\varepsilon$} node[pos=2/3,right]{\footnotesize $2\varepsilon$};
    \draw[red] (6,-3) -- (6,0) node[pos=1/3,right]{\footnotesize $2\varepsilon$} node[pos=5/6,right]{\footnotesize $\delta$};
    \draw (6,0) -- (6,4) node[midway,right]{\footnotesize $1$};
    \draw (5,3) -- (6,4) node[midway, above]{\footnotesize $\delta$};
    \draw(5,-2) -- (6,-3) node[midway, below]{\footnotesize $\varepsilon$};
        \node[circle,fill,inner sep=2pt] at (5,-2){};
        \node[circle,fill,inner sep=2pt] at (5,0){};
        \node[circle,fill,inner sep=2pt] at (5,1){};
        \node[circle,fill,inner sep=2pt] at (5,3){};
        \node[circle,fill,inner sep=2pt] at (6,4){};
        \node[circle,fill,inner sep=2pt] at (6,0){};
        \node[circle,fill,inner sep=2pt] at (6,-1){};
        \node[circle,fill,inner sep=2pt] at (6,-3){};

    \draw[dashed,decoration={markings, mark=at position 0.5 with {\arrow{>}}}, postaction={decorate}] (0,0) -- (1,0) node[midway,above]{d};
    \draw[dashed,decoration={markings, mark=at position 0.5 with {\arrow{>}}}, postaction={decorate}] (0,4) -- (1,4) node[midway,below]{b};
    \draw[dashed,decoration={markings, mark=at position 0.5 with {\arrow{>}}},postaction={decorate}] (1,0) -- (5,0) node[midway,above]{a};
    \draw[dashed,decoration={markings, mark=at position 0.5 with {\arrow{>}}},postaction={decorate}] (5,0) -- (6,0) node[midway,above]{e};
    \draw[dashed,decoration={markings, mark=at position 0.5 with {\arrow{>}}},postaction={decorate}] (1,4) -- (2,4) node[midway,below]{c};
    \end{tikzpicture}

    \caption{A view of the folded mapping torus $X$ with 1-cells $e$ labeled by $z_t(e)$, where $z_t$ is a positive cocycle. To define $z_t$ for $t < 1$, we choose $0 < \eps < \min\{(1-t)/3,1/2\}$ and $\delta = 1 - t - 3\eps$, so that $\eps,\delta,1-2\eps > 0$. In the proof of \Cref{lem:brown-alg-equality}, we show $z_t$ represents the class $r^* + tb^*$.}
    \label{fig:positive-cocycles}
\end{figure}

For the purposes of proving \Cref{thm:extended-example}, it would suffice to check the below result for our running example, which can be done directly by analyzing \Cref{fig:folded-mapping-torus} or \Cref{fig:folded-mapping-torus-cells}. The general statement may perhaps be of independent interest, so we prove it here.

\begin{proposition} \label{prop:skew-edge-nontrivial-loop}
    For any lone axis automorphism $\varphi$ of $F_N$ with folded mapping torus $X$, the union of the skew 1-cells of $X$ is a homologically nontrivial loop.
\end{proposition}

\begin{proof}
    Let $f: \Gamma\to \Gamma$ be a train-track representative of $\varphi$. Consider a Stallings fold decomposition of $f$:
    \[
        \Gamma\xrightarrow{\pi} \Gamma_0 \xrightarrow{q_1} \Gamma_1 \xrightarrow{q_2} \dots \xrightarrow{q_k} \Gamma_k \xrightarrow{h} \Gamma.
    \]
    Each $q_i$ is a single fold and the final map $h$ is a homeomorphism. For $0\leq i \leq k$ the graph $\Gamma_i$ has a vertex set of $(hq_k\dots q_{i+1})\inv(V(\Gamma))$ and edges labeled by the edges of $\Gamma$. Each graph $\Gamma_i$ is a train track graph for $\varphi$, with train track representative $(q_i\dots q_1)\pi h(q_k\dots q_{i+1})$. Since $\varphi$ has a lone axis, by \Cref{lem:mp16-lone-axis-unique-illegal-turn} each $\Gamma_i$ has a unique illegal turn. Because the illegal turn is unique, this turn is mapped to a degenerate turn, so there is an illegal turn at a vertex $v_i\in V(\Gamma_i)$ if and only if two edges with the same oriented label are incident to $v_i$.

    We note that the graphs $\Gamma_i$ are embedded `horizontally' in the folded mapping torus $X$ and that the semiflow $\psi$ takes $\Gamma_i$ to $\Gamma_{i+1}$ exactly as the map $q_{i+1}$. The associated skew 1-cell $s_{i+1}$ of $X$ begins at the vertex $v$ of $\Gamma_i$ which had the illegal turn and ends at a vertex $w$ of $\Gamma_{i+1}$ that is the $q_{i+1}$ image of two distinct vertices of $\Gamma_i$. For the union of the skew 1-cells to form a loop, we must have that the illegal turn of $\Gamma_{i+1}$ occurs at the vertex $w$, so that the next skew 1-cell $s_{i+2}$ starts at the endpoint of $s_{i+1}$. We will show that this is the case in the language of the Stallings fold decomposition.
    
    For simplicity, consider $q_1: \Gamma_0 \to \Gamma_1$; the general case will follow in the same manner. Let $v$ be the vertex of $\Gamma_0$ with an illegal turn $\{a_1, a_2\}$. On the level of vertices, the restricted map $q_1|_{V(\Gamma_0)}: V(\Gamma_0)\to V(\Gamma_1)$ is a surjection where the only failure of injectivity is the identification of the two vertices $w_1, w_2\in V(\Gamma_0)$ which are the opposite endpoints of the folded edges $a_1$ and $a_2$. Let $w = q_1(w_1) = q_1(w_2)\in V(\Gamma_1)$ and let $u\in V(\Gamma_1)$ be the vertex with the illegal turn of $\Gamma_1$, which we call $\{e_1,e_2\}$. Again, note that $e_1$ and $e_2$ carry the same oriented edge label. Assume for the sake of contradiction that $u\neq w$: we will show that this implies there are two illegal turns in $\Gamma_0$, contradicting that $\varphi$ has a lone axis via \Cref{lem:mp16-lone-axis-unique-illegal-turn}.

    Case 1: $u\neq w$, $u\neq q_1(v)$. Let $u_0 = q_1\inv(u)\in \Gamma_0$, this is well-defined because $u\neq w$. Since $u$ is neither $q_1(v)$ nor $w$, $q_1$ is a local homeomorphism around $u_0$ preserving the edge labels. Since $u$ has an illegal turn $\{e_1,e_2\}$, this can be pulled back to give an illegal turn $\{q_1\inv(e_1), q_1\inv(e_2)\}$ of $u_0$ in $\Gamma_0$. Thus $\Gamma_0$ has two illegal turns, contradicting the lone axis assumption.

    Case 2: $u = q_1(v) \neq w$. Now, there are edges $e_1,e_2$ incident to $u$ carrying identical edge labels. Let $b_1,b_2$ be edges of $\Gamma_0$ such that $q_1(b_1) = e_1$ and $q_1(b_2) = e_2$. One of these choices may not be unique if, say, $e_1$ is the image of the folded edges $a_1$ and $a_2$. If so, there are two illegal turns at $v$, namely $\{a_1, a_2\}$ and $\{a_1, b_2\}$.
    If there are unique choices of $b_1$ and $b_2$, there are still two illegal turns at $v$: the turn $\{b_1, b_2\}$ which is the preimage of $\{e_1,e_2\}$ and the the original turn where the fold was performed $\{a_1,a_2\}$. Again, this contradicts the lone axis assumption and we have that $s$ must be a loop.

    Finally, the loop  $s$ is nontrivial in homology as it is not freely homotopic into the base graph $\Gamma\subset X$. In $H_1(X)$, $s = ta+r$ for some $a\in H_1(\Gamma)$ and $t\in\ZZ$.
\end{proof}

\begin{lemma}
    For the mapping torus group $G$ of $\varphi$ from \Cref{thm:extended-example}, all primitive integral lone axis automorphisms in $C_s$ lie on the line $kb^* + (k+1) r^*$.
\end{lemma}

We note that for $k = 0$, this expression recovers $r^*$, which is associated to $\varphi$. The section corresponding to $b^* + 2r^*$ ($k=1$) is depicted in \Cref{fig:folded-mapping-torus}, and sections for arbitrary $k$ are shown in \Cref{fig:possible-lone-axis-sections}.

\begin{proof}
    Let $s$ be the skew loop guaranteed by \Cref{prop:skew-edge-nontrivial-loop}.  By \Cref{two-skew-intersections-illegal}, for a primitive integral class $\alpha\in \mathcal{P}\subseteq H^1(G; \RR)$ with monodromy automorphism $\varphi_\alpha$ having a lone axis, any corresponding section $\Theta_\alpha$ needs to intersect $s$ only once. This gives a cohomological constraint on the lone axis automorphisms within $\mathcal{P}$: $\alpha(s) = 1$. For our example, $s = r - b$ in homology, so $r^*(s) = 1$ and $b^*(s) = -1$. Hence the cohomology class $\alpha\in \mathcal{P}$ with $\alpha(s) = 1$ must have the form $kb^* + (k+1)r^*$.
\end{proof}

\begin{figure}
    \centering
    \begin{tikzpicture}[scale=1.7]

    \draw[dashed] (0,0) -- (6,0);
    \draw[dashed] (0,4) -- (2,4);

    \draw[red] (-1,1) -- (-1,3);
    \draw (-1,1) -- (0,0) node[midway,left]{\footnotesize $d_1$};
    \draw (-1,3) -- (0,2) node[midway,above]{\footnotesize $d_3$};
        \draw[dotted,thick] (-0.5,0.5) -- (-0.5,2.5);
        \draw[violet,thick] (-0.5,0.5) -- (0,0.5) node[midway,above]{\footnotesize $e_1$};
    
    \draw[blue] (0,0) -- (0,4) node[pos=1/4,left]{\footnotesize $\bar{u}_1$} node[pos=3/4,left]{\footnotesize $\bar{u}_2$};
    \draw[red] (0,4) -- (0,5) node[midway,left]{\footnotesize $\bar{v}_1$};
    \draw[red] (1,-1) -- (1,0);
    \draw (1,0) -- (1,4) node[pos=1/4,left]{\footnotesize $\bar{w}$};
    \draw (0,5) -- (1,4) node[midway,above]{\footnotesize $d_1$};
    \draw (0,0) -- (1,-1) node[midway,below]{\footnotesize $d_4$};
        \draw[violet,thick] (0,4.5) -- (0.5,4.5) node[midway,below]{\footnotesize $s_2$};
        \draw[violet, fill] (0,2.75) -- (0,3.25) -- (1,2.25) -- (1,1.75) -- (0,2.75);
        \draw[violet,thick] (0,0.5) -- (1,1.75) node[midway,above,sloped]{\footnotesize $e_{4,1}$};

        \draw[violet] (0,3.25) -- (1,2.25) node[pos=0.375,above,sloped]{\footnotesize $e_{4,2} \dots e_{4,k+1}$};

    \draw[blue] (1,4) -- (1,6) node[midway,left]{\footnotesize $\bar{u}_1$};
    \draw[blue] (5,-2) -- (5,0) node[midway,right]{\footnotesize $\bar{u}_2$};
    \draw (1,6) -- (2,4) node[pos=1/4,right]{\footnotesize $d_2$} node[pos=3/4,right]{\footnotesize $d_1$};
    \draw(2,4) -- (5,1) node[pos=1/6,above]{\footnotesize $d_4$} node[midway,above]{\footnotesize $d_3$} node[pos=5/6,above]{\footnotesize $d_2$};
    \draw (1,-1) -- (5,-2) node[midway,below]{$d_3$};
        \draw[dotted, thick] (3,-1.5) -- (3,3);
        \draw[violet,thick] (1,4.5) -- (1.75,4.5) node[midway,below]{\footnotesize $s_1$};
        \draw[violet,thick] (1,2.25) -- (5,0.5) node[pos=1/4,above] {\footnotesize $t_{k+1}$} node[pos=3/4,above]{\footnotesize $e_{3,k+1}$};
        \draw[violet,fill] (1,1.75) -- (1,2.25) -- (5,-0.75) -- (5,-1.25) -- (1,1.75);
        \draw[violet] (1,1.75) -- (5,-1.25) node[pos=1/4,sloped,below]{\footnotesize $t_1 \dots t_k$} node[pos = 3/4,sloped,below]{\footnotesize $e_{3,1}\dots e_{3,k}$};
    
    \draw[red] (5,0) -- (5,3) node[pos=1/6,right]{\footnotesize $\bar{v}_1$};
    \draw[red] (6,-3) -- (6,0);
    \draw (6,0) -- (6,4) node[midway,right]{\footnotesize $\bar{w}$};
    \draw (5,3) -- (6,4) node[midway,above]{\footnotesize $d_4$};
    \draw(5,-2) -- (6,-3) node[midway,below]{\footnotesize $d_2$};

        \draw[violet,thick] (5,0.5) -- (6,2.25) node[midway,above,sloped]{\footnotesize $e_{2,k+1}$};
        \draw[violet,fill] (5,-1.25) -- (5,-0.75) -- (6,2.25) -- (6,1.75) -- (5,-1.25);
        \draw[violet] (5,-1.25) -- (6,1.75) node[pos=1/2,sloped,below]{\footnotesize $e_{2,1}\dots e_{2,k}$};

        \node[violet,star,star points=5,fill,inner sep=2pt] at (-0.5,0.5){};
        \node[violet,star,star points=5,fill,inner sep=2pt] at (0.5,4.5){};
        \node[violet,star,star points=5,fill,inner sep=2pt] at (1.75,4.5){};
        \node[circle,fill,inner sep=2pt] at (-1,1){};
        \node[circle,fill,inner sep=2pt] at (-1,3){};
        \node[circle,fill,inner sep=2pt] at (0,0){};
        \node[circle,fill,inner sep=2pt] at (0,2){};
        \node[circle,fill,inner sep=2pt] at (0,4){};
        \node[circle,fill,inner sep=2pt] at (0,5){};
        \node[circle,fill,inner sep=2pt] at (1,-1){};
        \node[circle,fill,inner sep=2pt] at (1,0){};
        \node[circle,fill,inner sep=2pt] at (1,4){};
        \node[circle,fill,inner sep=2pt] at (1,6){};
        \node[circle,fill,inner sep=2pt] at (1.5,5){};
        \node[circle,fill,inner sep=2pt] at (2,4){};
        \node[circle,fill,inner sep=2pt] at (3,3){};
        \node[circle,fill,inner sep=2pt] at (4,2){};
        \node[circle,fill,inner sep=2pt] at (5,1){};
        \node[circle,fill,inner sep=2pt] at (5,0){};
        \node[circle,fill,inner sep=2pt] at (5,-2){};
        \node[circle,fill,inner sep=2pt] at (5,3){};
        \node[circle,fill,inner sep=2pt] at (6,4){};
        \node[circle,fill,inner sep=2pt] at (6,0){};
        \node[circle,fill,inner sep=2pt] at (6,-1){};
        \node[circle,fill,inner sep=2pt] at (6,-3){};
    \end{tikzpicture}
    \caption{Depiction of a general section graph $\Theta_k$ dual to the class $kb^* + (k+1)r^*$. The thick purple box segments represent $k$ parallel edges, and the thinner purple edges represent single edges. There are vertices of $\Theta_k$ where it intersects 1-cells of $X$ and also where the section meets the vertical dotted line; this ensures that the first return map sends vertices to vertices. The edges of $\Theta_k$ are oriented from left to right, and where $k$ parallel edges are labeled, they are listed from the bottom edge to the top edge.}
    \label{fig:possible-lone-axis-sections}
\end{figure}
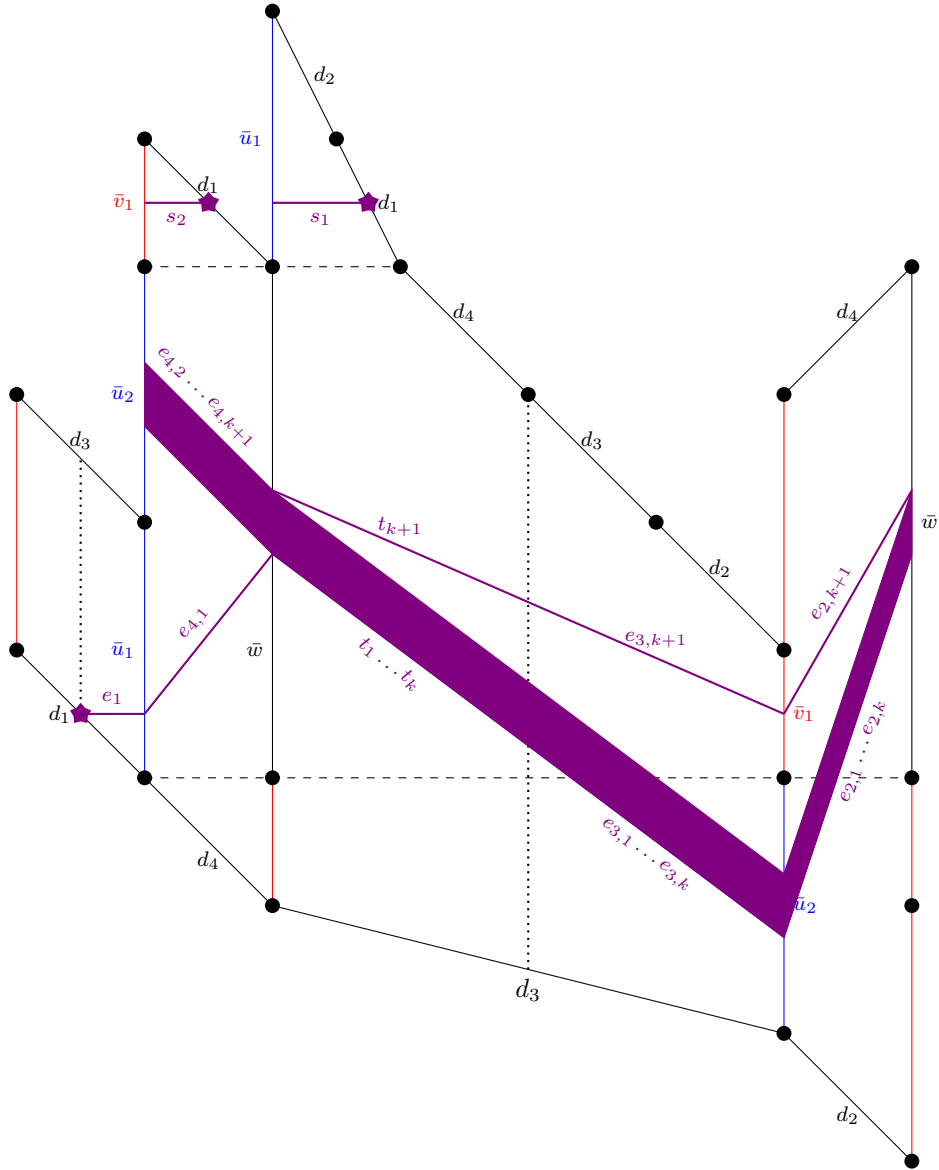

Now we proceed to analyze the sections corresponding to the cohomology classes $\alpha_k = kb^* + (k+1)r^*$ for $k\geq 1$.

\begin{lemma}\label{lem:particular-automorphism-representation}
    Let $k\in \mathbb{N}$ and $\langle s_1,s_2,t_1,\dots,t_{k+1}\rangle\cong F_{k+3}$. The automorphism $\varphi_k$ defined by
    \begin{align*}
        s_1 & \mapsto t_1 \\
        s_2 & \mapsto s_2 t_1 \\
        t_i & \mapsto t_{i+1} \ \text{for $1\leq i \leq k$} \\
        t_{k+1} & \mapsto s_2 s_1 t_1 s_2\inv \\
    \end{align*}
    is in the outer class of automorphisms determined by the cohomology class $\alpha_k = kb^* + (k+1) r^* \in H^1(G;\RR)$, where $G$ is the group defined in \Cref{thm:extended-example}.
\end{lemma}
\begin{proof}
    Using multiples of the positive cocycles shown in \Cref{fig:positive-cocycles}, we construct section graphs $\Theta_k$ which we depict in \Cref{fig:possible-lone-axis-sections}. \Cref{fig:theta-k-abstract-graph} shows $\Theta_k$ as an abstract graph. We pick the basepoint of $\Theta_k$ to be the intersection point of $\Theta_k$ with $d_1$, and choose a maximal tree $T_k\subset \Theta_k$ consisting of the edges $e_1,e_{2,i},e_{3,i},e_{4,i}$ for $1\leq i \leq k+1$. Let $\kappa:\Theta_k\to R_{k+3}$ be the map collapsing the maximal tree; this map is a homotopy inverse to the marking $\iota: R_{k+3}\to \Theta_k$ we will describe later. In this way, the generators of $\pi_1(\Theta_k)\cong F_{k+3}$ correspond to the $k+3$ edges of $\Theta_k$ not in the maximal tree: $s_1,s_2, \text{ and } t_i$ for $1\leq i \leq k+1$. We can determine the automorphism by analyzing the first return map of this section.

    To understand the first return map, we analyze \Cref{fig:possible-lone-axis-sections}. For the edges $t_1,\dots, t_k$, the semiflow brings the edge back to the graph within the same 2-cell of $X$, so we can see that $t_i\mapsto t_{i+1}$ for $1\leq i\leq k$, and most of the edges have similar behavior. For some edges, we have to analyze the identifications of $X$ more closely. For instance the edge $e_{2,k+1}$ flows onto the 1-cell $d_4$ then further flows to the edge $e_{4,1}$ of $\Theta_k$. The edge $s_2$ flows onto the `upper half' of $d_1$, which then flows to the upper half of $d_3$, then onto the edge $t_1$. The full edge map is summarized in the following table.
    \begin{center}
    \begin{tabular}{|l|l|l|}
        \hline 
        $e\in E(\Theta_k)$ & $f_k(e)$ & $\kappa(f_k(e))$ \\
        \hline 
        $e_1$ & $e_{3,1}$ &  \\
        $e_{j,i},2\leq j\leq 4,1\leq i\leq k$ & $e_{j,i+1}$ &  \\
        $e_{2,k+1}$ & $\bar e_{4,1}$ &  \\
        $e_{3,k+1}$ & $t_1 e_{3,1} e_{2,1}$ & $t_1$ \\
        $e_{4,k+1}$ & $s_2 e_1$ & $s_2$ \\
        $s_1$ & $e_{2,1} t_1$ & $t_1$ \\
        $s_2$ & $t_1$ & $t_1$ \\
        $t_i, 1\leq i\leq k$ & $t_{i+1}$ & $t_{i+1}$ \\
        $t_{k+1}$ & $s_1e_1 e_{4,1}$ & $s_1$ \\
        \hline
    \end{tabular}
    \end{center}

    To obtain the automorphism, we need to understand the marking $\iota$ we have implicitly defined with our choice of basepoint and maximal tree. To begin, note that the loop $e_1 s_1$ represents the element $s_1\in F_{k+3}$. We claim that the loop $e_1 e_{4,1} t_1 e_{3,1} e_{2,1} \bar e_{4,1} \bar e_1$ represents the element $t_1$. To represent the loops which are `higher' in $X$, we have to describe certain `upward paths' through the maximal tree.

    Note that the vertical 1-cell $\bar{w}$ of \Cref{fig:possible-lone-axis-sections} intersects $\Theta_k$ at $k+1$ points, which we will call $w_1,\dots,w_{k+1}$, ordered from bottom to top with respect to the semiflow. For any $1\leq i\leq k$, the edge path $\bar e_{2,i} e_{4,i+1}$ begins at $w_i$ and ends at $w_{i+1}$, for $i=1,k$, these can be seen connecting distinct black vertices in \Cref{fig:theta-k-abstract-graph}. We define edge paths $p_1 = e_1 e_{4,1}$ and for $2\leq i\leq k+1$, $p_i = (e_1 e_{4,1}) (\bar e_{2,1} e_{4,2} \dots \bar e_{2,i-1} e_{4,i})$. The path $p_i$ is the path through the maximal tree beginning at the basepoint and ending at $w_{i}$. We use this notation to describe the marking.

    \begin{center}
    \begin{tabular}{|l|l|}
        \hline
        $a\in F_{k+3}$ & $\iota(a)$ \\
        \hline
        $s_1$ & $e_1 s_1$ \\
        $s_2$ & $p_{k+1} \bar e_{2,k+1} s_2$  \\
        $t_i$ & $p_{i} t_i e_{3,i} e_{2,i} \bar p_{i}$ \\
        \hline
    \end{tabular}
    \end{center}

    We can now use the data of the above two tables to obtain the automorphism $\varphi_k = \kappa f_k \iota$. We note that $f_k(p_i)$ lies within the tree $T_k$ for all $1\leq i \leq k$, and $f_k(p_{k+1})$ crosses only the edge $s_2$, which comes from the $e_{4,k+1}$ edge at the very end of the edge path $p_{k+1}$. That is, $\kappa(f_k(p_{k+1})) = s_2$. The automorphism $\varphi_k$ factored as $(\kappa f_k)(\iota)$ is shown below:

    \begin{align*}
        s_1 & \xrightarrow{\iota} e_1s_1 \xrightarrow{\kappa f_k} t_1 \\
        s_2 & \xrightarrow{\iota} p_{k+1} \bar e_{2,k+1} s_2 \xrightarrow{\kappa f_k} s_2 t_1 \\
        t_i & \xrightarrow{\iota} p_i t_i e_{3,i} e_{2,i} \bar p_i \xrightarrow{\kappa f_k} t_{i+1} \\
        t_{k+1} & \xrightarrow{\iota} p_{k+1} t_{k+1} e_{3,k+1} e_{2,k+1} \bar p_{k+1} \xrightarrow{\kappa f_k} s_2s_1t_1 s_2\inv. \qedhere
    \end{align*}
\end{proof}

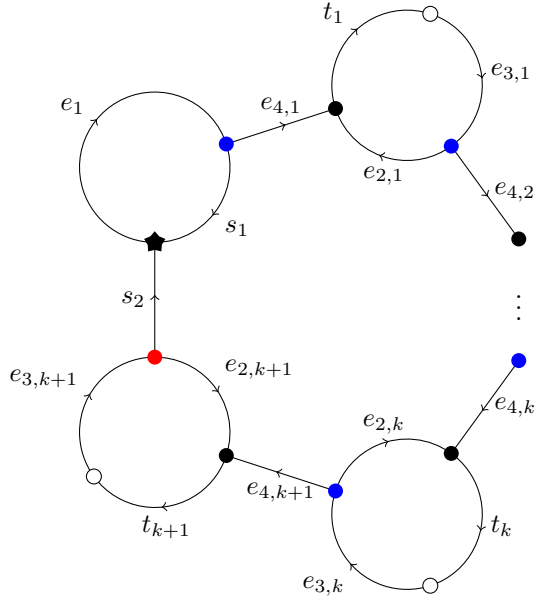
\begin{figure}
    \centering

    \begin{tikzpicture}[
        scale=1,
        opendot/.style={
        circle,
        fill=white,
        draw,
        outer sep=0pt,
        inner sep=2pt
        },
        dot/.style={
        circle,
        fill,
        inner sep=2pt,
        },
        backwardarrow/.style={
        decoration={markings,mark=at position 0.5 with {\arrow{<}}},
        postaction={decorate}
        },
        forwardarrow/.style={
        decoration={markings,mark=at position 0.5 with {\arrow{>}}},
        postaction={decorate}
        }
    ]
        \draw[backwardarrow] ({3*cos(72)+cos(-54)},{3*sin(72)+sin(-54)}) arc [start angle = -54, end angle = 72, radius = 1] node[midway,right]{$e_{3,1}$};
        \draw[backwardarrow] ({3*cos(72)+cos(72)},{3*sin(72)+sin(72)}) arc [start angle = 72, end angle = 198, radius = 1] node[midway,above left]{$t_1$};
        \draw[backwardarrow] ({3*cos(72)+cos(198)},{3*sin(72)+sin(198)}) arc [start angle = 198, end angle = 306, radius = 1] node[midway,below]{$e_{2,1}$};
        
        \draw[backwardarrow] ({3*cos(2*72)+cos(18)}, {3*sin(2*72)+sin(18)}) arc [start angle = 18, end angle = 270,radius = 1cm] node[midway,above left]{$e_1$};
        \draw[backwardarrow] ({3*cos(2*72)}, {3*sin(2*72)-1}) arc [start angle = -90, end angle = 18, radius = 1cm] node[midway,below right]{$s_1$};
        
        \draw[backwardarrow] ({3*cos(3*72)},{3*sin(3*72)+1}) arc [start angle = 90, end angle = 216,radius=1cm] node[midway,above left]{$e_{3,k+1}$};
        \draw[backwardarrow] ({4*cos(216)},{4*sin(216)}) arc[start angle = 216, end angle = 342,radius=1cm] node[midway, below]{$t_{k+1}$};
        \draw[backwardarrow] ({3*cos(3*72)+cos(-18)},{3*sin(3*72)+sin(-18)}) arc[start angle = -18, end angle = 90,radius = 1cm] node[midway, above right]{$e_{2,k+1}$};
        
        \draw[backwardarrow] ({3*cos(4*72)+cos(162)},{3*sin(4*72)+sin(162)}) arc[start angle = 162, end angle = 288, radius = 1cm] node[midway, below left]{$e_{3,k}$};
        \draw[backwardarrow] ({4*cos(288)},{4*sin(288)}) arc[start angle = -72, end angle = 54, radius = 1cm] node[midway,right]{$t_k$};
        \draw[backwardarrow] ({3*cos(288)+cos(54)},{3*sin(288)+sin(54)}) arc[start angle = 54, end angle = 162, radius = 1cm] node[midway, above]{$e_{2,k}$};

        \draw[backwardarrow] ({3+cos(126)},{sin(126)}) -- ({3*cos(72)+cos(-54)},{3*sin(72)+sin(-54)}) node[midway,right]{$e_{4,2}$};
        \draw[backwardarrow] ({3*cos(72)+cos(-162)},{3*sin(72)+sin(-162)}) -- ({3*cos(2*72)+cos(18)},{3*sin(2*72)+sin(18)}) node[midway,above]{$e_{4,1}$};
        \draw[backwardarrow] ({3*cos(2*72)},{3*sin(2*72)-1}) -- ({3*cos(3*72)},{3*sin(3*72)+1}){} node[midway,left]{$s_2$};
        \draw[backwardarrow] ({3*cos(3*72)+cos(-18)},{3*sin(3*72)+sin(-18)}){} -- ({3*cos(4*72)+cos(162)},{3*sin(4*72)+sin(162)}){} node[midway,below]{$e_{4,k+1}$};
        \draw[backwardarrow] ({3*cos(4*72)+cos(54)},{3*sin(4*72)+sin(54)}){} -- ({3+cos(234)},{sin(234)}){} node[midway,right]{$e_{4,k}$};

        \node[dot] at ({3+cos(126)},{sin(126)}){};
        \node[blue,dot] at ({3+cos(234)},{sin(234)}){};
        \path ({3+cos(126)},{sin(126)}) -- node[auto=false]{\vdots} ({3+cos(234)},{sin(234)});

        \node[blue,dot] at ({3*cos(72)+cos(-54)},{3*sin(72)+sin(-54)}){};
        \node[dot] at ({3*cos(72)+cos(-162)},{3*sin(72)+sin(-162)}){};
        \node[opendot] at ({4*cos(72)},{4*sin(72)}){};

        \node[blue,dot] at ({3*cos(2*72)+cos(18)},{3*sin(2*72)+sin(18)}){};
        \node[star,star points=5,fill,inner sep=2pt] at ({3*cos(2*72)},{3*sin(2*72)-1}){};

        \node[red,dot] at ({3*cos(3*72)},{3*sin(3*72)+1}){};
        \node[dot] at ({3*cos(3*72)+cos(-18)},{3*sin(3*72)+sin(-18)}){};
        \node[opendot] at ({4*cos(3*72)},{4*sin(3*72)}){};

        \node[blue,dot] at ({3*cos(4*72)+cos(162)},{3*sin(4*72)+sin(162)}){};
        \node[dot] at ({3*cos(4*72)+cos(54)},{3*sin(4*72)+sin(54)}){};
        \node[opendot] at ({4*cos(4*72)},{4*sin(4*72)}){};

    \end{tikzpicture}
    \caption{
        The section $\Theta_k$ of \Cref{fig:possible-lone-axis-sections} and \Cref{lem:particular-automorphism-representation} depicted as an abstract graph. The maximal tree chosen in \Cref{lem:particular-automorphism-representation} consists of all `$e$' edges, letting the generators of $\pi_1(\Theta_k)$ be $s_1, s_2, t_i$ for $1\leq i \leq k+1$. The colorings of the vertices match the coloring of the intersecting 1-cell in \Cref{fig:possible-lone-axis-sections}, the star is the intersection with the skew 1-cell and is the chosen basepoint. 
    }
    \label{fig:theta-k-abstract-graph}
\end{figure}

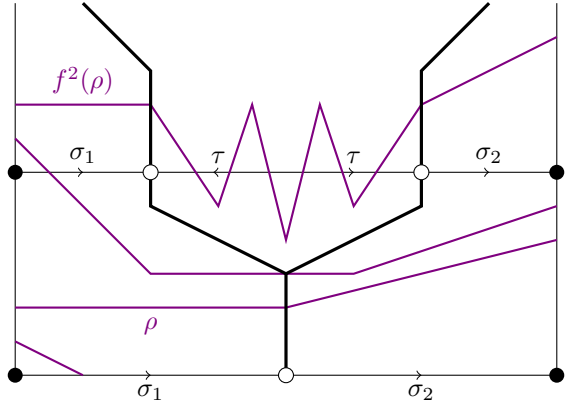
\begin{figure}
    \centering
    \begin{tikzpicture}[
    scale=0.9,
    dot/.style={
        circle,
        fill,
        inner sep=2pt,
        },
        opendot/.style={
        circle,
        fill=white,
        draw,
        outer sep=0pt,
        inner sep=2pt
        },
        backwardarrow/.style={
        decoration={markings,mark=at position 0.5 with {\arrow{<}}},
        postaction={decorate}
        },
        forwardarrow/.style={
        decoration={markings,mark=at position 0.5 with {\arrow{>}}},
        postaction={decorate}
        }
    ]
        \draw[forwardarrow] (-4,0) -- (0,0) node[midway,below]{$\sigma_1$};
        \draw[forwardarrow] (0,0) -- (4,0) node[midway,below]{$\sigma_2$};

        \draw (-4,0) -- (-4,5.5);
        \draw (4,0) -- (4,5.5);
        
        \draw[forwardarrow] (-4,3) -- (-2,3) node[midway,above]{$\sigma_1$};
        \draw[backwardarrow] (-2,3) -- (0,3) node[midway,above]{$\tau$};
        \draw[forwardarrow] (0,3) -- (2,3) node[midway,above]{$\tau$};
        \draw[forwardarrow] (2,3) -- (4,3) node[midway,above]{$\sigma_2$};

        \draw[violet,thick] (-4,1) -- (0,1)node[pos=1/2,below]{$\rho$} -- (4,2);

        \draw[violet,thick] (-4,4) -- (-2,4) node[midway,above]{$f^2(\rho)$} -- (-1,2.5) -- (-0.5,4) -- (0,2) -- (0.5,4) -- (1,2.5) -- (2,4) -- (4,5);

        \draw[violet,thick] (-4,3.5) -- (-2,1.5) -- (1,1.5) -- (4,2.5);
        \draw[violet,thick] (-4,0.5) -- (-3,0);

        \draw[very thick] (0,0) -- (0,1.5);
        \draw[very thick] (0,1.5) -- (-2,2.5) -- (-2,4.5) -- (-3,5.5);
        \draw[very thick] (0,1.5) -- (2,2.5) -- (2,4.5) -- (3,5.5);

        \node[dot] at (-4,0){};
        \node[dot] at (4,0){};
        \node[opendot] at (0,0){};

        \node[dot] at (-4,3){};
        \node[dot] at (4,3){};
        \node[opendot] at (-2,3){};
        \node[opendot] at (2,3){};        
        
    \end{tikzpicture}
    \caption{The domain of the map $F:[-1,1]\times[0,\infty)\to X$ defined in the proof of \Cref{prop:pnp-invariance}. The preimage $F\inv(\Gamma)$ is in purple, and the thick black segments are preimages of selected 1-cells of $X$. For $(s,t)$ above the thick black segments, we have that $F(-s,t) = F(s,t)$.}
    \label{fig:pnp-invariance}
\end{figure}

We wish to analyze the ideal Whitehead graphs of the monodromies $\varphi_k$ produced in the above lemma. To do so with the description of the ideal Whitehead graph we provided, we must guarantee that none of the cross sections $\Theta_\alpha$ admit a periodic Nielsen path. Roughly, flowing a periodic Nielsen path of $\Theta_\alpha$ forward in $X$ yields a periodic Nielsen path in $\Gamma$ and vice versa. We will use a reparametrization $\psi^\alpha$ of $\psi$ that is the first return map of $\Theta_\alpha$ at time one, see \cite{DKL1,DKL2}. We note that the following proposition was suggested by Dowdall--Kapovich--Leininger, but does not appear in their work.

\begin{proposition} \label{prop:pnp-invariance}
    Let $f:\Gamma\to\Gamma$ be an expanding irreducible train track map and $X$ a folded mapping torus of $f$. Let $\alpha\in \mathcal{P}$ be dual to a cross section $\Theta_\alpha$ with first return map $f_\alpha:\Theta_\alpha\to\Theta_\alpha$. Then, $f$ admits a periodic indivisible Nielsen path if and only if $f_\alpha$ admits a periodic indivisible Nielsen path.
\end{proposition}
\begin{proof}
    Let $\sigma: [-1,1]\to\Theta_\alpha$ be a periodic indivisible Nielsen path, i.e. $f_\alpha^n(\sigma)$ is homotopic to $\sigma$ rel endpoints. Note that we can write $\psi^\alpha_n(\sigma) = f^n_\alpha(\sigma)$. By \cite[Lemma 3.4]{BH92}, $\sigma$ can be written as a concatenation of legal paths $\sigma_1\sigma_2$ such that $f^n_\alpha(\sigma_1) = \sigma_1\bar\tau$ and $f^n_\alpha(\sigma_2) = \tau \sigma_2$ for some path $\tau$. Let $\tau$ have parametrization $\tau:[0,1]\to \Theta_\alpha$. We parametrize $\sigma$ so that it restricts to $\sigma_1$ on $[-1,0]$ and $\sigma_2$ on $[0,1]$ and additionally require
    \[
        f_\alpha^n(\sigma(s)) = \begin{cases}
            \sigma(2s+1) & s\in [-1,-\frac{1}{2}] \\
            \tau(2|s|) & s\in [-\frac{1}{2},\frac{1}{2}] \\
            \sigma(2s-1) & s\in [\frac{1}{2},1].
        \end{cases}
    \]
    We define a map $F: [-1,1]\times [0,\infty)\to X$ by $F(s,t) = \psi^\alpha_t(\sigma(s))$. Note that $F(-1,t)$ and $F(1,t)$ are periodic orbits of $\psi$ since $\sigma$ is a periodic Nielsen path. Consider $F\inv(\Gamma)\subset [-1,1]\times [0,\infty)$. Since $\Gamma$ is transverse to $\psi_\alpha$, $F\inv(\Gamma)$ is a disjoint union of arcs. Finitely many of these arcs intersect $[-1,1]\times \{0\}$, the rest connect points of $\{\pm 1\}\times [0,\infty)$, see \Cref{fig:pnp-invariance}.

    Let $\hat\rho\subseteq F\inv(\Gamma)$ be an arc connecting $(-1,t^-)$ and $(1,t^+)$, parametrized so that $\hat\rho(s) = (s,h(s))$, where $h: [-1,1]\to\RR_+$ is a positive continuous function. Such a choice of parametrization can be made since $\Gamma$ is transverse to $\psi^\alpha$. Then $\rho(s) = F(\hat\rho(s)) = F(s,h(s))$ parametrizes an arc in $\Gamma$ connecting $v^- = F(-1,t^-)$ and $v^+ = F(1,t^+)$. We will show that $\rho$ is a periodic Nielsen path of $\Gamma$. 

    For $s\in[-1,-1/2]$, by our choice of parametrization of $\sigma$,
    \begin{align*}
        F(s,n+h(2s+1)) & = \psi^\alpha_{n+h(2s+1)} (\sigma(s)) \\
        & = \psi^\alpha_{h(2s+1)} f^n_\alpha(\sigma(s)) \\
        & = \psi^\alpha_{h(2s+1)} (\sigma(2s+1)) \\
        & = F(2s+1,h(2s+1)) \\
        & = F(\hat\rho(2s+1))\in\Gamma.
    \end{align*}
    A similar calculation shows that for $s\in [1/2,1]$, $F(s,n+h(2s-1)) = F(\hat\rho(2s-1))$. In particular, $(-1,n+h(-1))$ and $(1,n+h(1))$ are in $F\inv(\Gamma)$. There is an arc $\hat\rho_1\subset F\inv (\Gamma)$ connecting these points, which has a parametrization $\hat\rho_1(s) = (s,h_1(s))$ for some positive continuous function $h_1$. We have that
    \[
        F(\hat\rho_1(-\frac{1}{2})) = F(-\frac{1}{2},n+h(0)) = F(\hat\rho(0)) = F(\frac{1}{2},n+h(0)) =  F(\hat\rho_1(\frac{1}{2})).
    \]
    Thus, the path $\gamma: [-1/2,1/2]\to\Gamma$ defined by $\gamma(s) = F(s,h_1(s))$ is a loop. We show that $\gamma$ is nullhomotopic in $X$. Define a (free) homotopy $H: [-1/2,1/2]\times [0,1]\to X$ by $H(s,t) = F(s,(1-t)h_1(s)+nt))$, note that $H(-1/2,t) = H(1/2,t)$ for all $t\in[0,1]$. Then,
    \[
        H(s,0) = F(s,h_1(s)) = \gamma(s) \quad \text{and} \quad H(s,1) = F(s,n) = \tau(2|s|).
    \]
    Thus, $\gamma$ is freely homotopic in $X$ to $\bar\tau\tau$, so $\gamma$ is nullhomotopic in $X$. The inclusion of $\Gamma$ into $X$ is $\pi_1$-injective, so $\gamma$ is nullhomotopic in $\Gamma$. Then $F(\hat\rho_1)$ is homotopic to $\rho$. On the other hand, $F(\hat\rho)$ is a forward flow of $\rho$, so $F(\hat\rho) = f^m(\rho)$ for some $m$.

    The converse direction is essentially the same, reversing the roles of $\rho$ and $\sigma$, using the `standard' parametrization of $\psi$ instead of $\psi^\alpha$, and appealing to Theorem A of \cite{DKL1} for the $\pi_1$-injectivity of the inclusion $\Theta_\alpha\to X$.
\end{proof}
    
\begin{proof}[Proof of \Cref{thm:extended-example}]
    We will show that the automorphisms $\varphi_k$ satisfy the lone axis criteria of \Cref{thm:mp16-lone-axis-criteria}. We analyze the train track map $f_k: \Theta_k\to \Theta_k$.

    As noted in \Cref{ex:lone-axis}, $f: \Gamma\to\Gamma$ admits no periodic Nielsen paths by \cite{coulbois-train-track}. By \Cref{prop:pnp-invariance}, this implies that $f_k: \Theta_k\to \Theta_k$ admits no periodic Nielsen paths. Thus we may use the description of the ideal Whitehead graph provided in \Cref{sec:axis-bundle} to compute $\mathcal{IW}(\varphi_k)$.

    The periodic vertices of $\Theta_k$ are the $2k+3$ vertices colored black, red, or blue in \Cref{fig:theta-k-abstract-graph}, which are permuted transitively under $f$. Note that the basepoint is not periodic. Starting at the direction $s_1$, we analyze the direction map $Df$:
    \begin{align*}
        s_1 & \mapsto e_{2,1} \mapsto \dots \mapsto e_{2,k+1} \mapsto \bar e_{4,1}\mapsto \dots \mapsto \bar e_{4,k+1} \mapsto \bar e_1 \\
        \bar e_1 & \mapsto \bar e_{3,1} \mapsto \dots \mapsto \bar e_{3,k+1} \mapsto \bar e_{2,1} \mapsto \dots \mapsto \bar e_{2,k+1} \mapsto e_{4,1} \\
        e_{4,1} & \mapsto \dots \mapsto e_{4,k+1} \mapsto s_2 \mapsto t_1\mapsto \dots \mapsto t_{k+1} \mapsto s_1.
    \end{align*}
    Notice that $Df$ transitively permutes all of the directions at periodic vertices and that $Df^{2k+3}$ transitively permutes the three directions at each periodic vertex (for instance, the directions $s_1, \bar e_1, e_{4,1}$). This implies that each periodic vertex is principal. The induced map on turns is then also a transitive permutation. We note that the turn $\{\bar e_1, e_{4,1}\}$ is taken, and so under powers of $f$, every turn between periodic vertices is taken. This implies that the ideal Whitehead graph $\mathcal{IW}(\varphi_k)$ is a disjoint union of $2k+3$ triangles, which do not have cut vertices. Finally, the rotationless index of $\varphi$ is
\[
    i(\varphi) = \sum_{j=1}^{2k+3} (1-\frac{3}{2}) = \frac{3}{2} - (k+3).
\]
An appeal to \Cref{thm:mp16-lone-axis-criteria} concludes the proof.
\end{proof}

\printbibliography

\end{document}